\documentclass[a4paper]{article}

\usepackage{amsmath}
\usepackage{amssymb}
\usepackage{amsthm}
\usepackage[UKenglish]{babel}
\usepackage{Baskervaldx}
\usepackage{bm}
\usepackage{booktabs}
\usepackage{breakurl}
\usepackage{cite}
\usepackage{multirow}
\usepackage{tabularx}
	\newcolumntype{Y}{>{\centering\arraybackslash}X}
\usepackage[normalem]{ulem}
\usepackage{url}
\usepackage[dvipsnames]{xcolor}
	\definecolor{linkblue}{HTML}{3d25b9}

\usepackage{tikz}
	\usetikzlibrary{arrows.meta}
	\usetikzlibrary{calc}
	
\usepackage{enumitem}
	\setlist{topsep=0pt,itemsep=0pt}
\usepackage[hang,flushmargin]{footmisc}
\usepackage[margin=25mm]{geometry}
\usepackage{microtype}
\usepackage{sectsty}
	\sectionfont{\large}
	\subsectionfont{\normalsize}
\usepackage{titlesec}
	\titlespacing{\section}{0pt}{12pt}{0pt}
	\titlespacing{\subsection}{0pt}{6pt}{0pt}
\usepackage{titling}
\usepackage[nottoc]{tocbibind} 
\usepackage{tocloft}

\usepackage[colorlinks=true,linkcolor=linkblue,citecolor=linkblue,urlcolor=magenta]{hyperref}
	\hypersetup{breaklinks=true}
\usepackage[capitalise,noabbrev]{cleveref}
	\crefname{equation}{equation}{equations}

\theoremstyle{plain}
	\newtheorem{theorem}{Theorem}
	\newtheorem{proposition}[theorem]{Proposition}
	
	\newtheorem{lemma}[theorem]{Lemma}

	\numberwithin{theorem}{section}
\theoremstyle{definition}

\newcommand{\dd}{\mathrm{d}}
\newcommand{\zq}{\zeta_q}
\newcommand{\zqm}{\zeta_q^M}

\newcommand{\pp}{\bm{p}}
\newcommand{\zz}{\bm{z}}
\newcommand{\LL}{\bm{L}}
\newcommand{\C}{\mathcal{C}}
\renewcommand{\geq}{\geqslant}

\renewcommand{\leq}{\leqslant}

\newcommand{\ket}[1]{| \, #1 \, \rangle}

\DeclareMathOperator{\tr}{Tr}
\newcommand{\bc}{\mathbb{C}}
\newcommand{\bn}{\mathbb{N}}
\newcommand{\bq}{\mathbb{Q}}
\newcommand{\cl}{\mathcal{L}}
\newcommand{\cm}{\mathcal{M}}
\newcommand{\cf}{\mathcal{F}}
\newcommand{\modm}{\cal M}
\newcommand{\Res}{\mathop{\,\rm Res\,}}

\tikzset{c/.style={every coordinate/.try}}

\let\OLDthebibliography\thebibliography
\renewcommand\thebibliography[1]{
 \OLDthebibliography{#1}
 \setlength{\parskip}{0pt}
 \setlength{\itemsep}{4pt plus 0.3ex}
}

\title{From double-scaled SYK correlators to Weil--Petersson volumes}
\author{Norman Do \and Paul Norbury}

\begin{document}

\makeatletter
\textbf{\large \thetitle}

\textbf{\theauthor}
\makeatother

School of Mathematics, Monash University, VIC 3800 Australia \\
Email: \href{mailto:norm.do@monash.edu}{norm.do@monash.edu}

School of Mathematics and Statistics, The University of Melbourne, VIC 3010 Australia \\
Email: \href{mailto:norbury@ms.unimelb.edu.au}{norbury@ms.unimelb.edu.au}

{\em Abstract.} Okuyama introduced a family of polynomials, whose coefficients depend on a parameter $q$, in his study of correlators in the double-scaled SYK model. He verified in small cases that their coefficients can be expressed in terms of certain $q$-zeta values and that the polynomials recover the Weil--Petersson volumes of moduli spaces studied by Mirzakhani under a certain $q \to 1$ limit. In this paper, we provide mathematically rigorous proofs of these two phenomena. The authors previously defined natural $q$-deformations of the Weil--Petersson volumes of moduli spaces of curves. We prove that these polynomials appear as the top degree part of Okuyama's polynomials. Our work provides a link between the two topics of the title, which hints at a ``quantum'' Weil--Petersson geometry and a combinatorial-geometric approach to double-scaled SYK correlators.

\emph{Acknowledgements.} The first author was supported by the Australian Research Council grant FT240100795.


\vspace{6mm} \hrule \vspace{4mm}

\tableofcontents

\vspace{6mm} \hrule

\section{Introduction}

Motivated by the calculation of correlators in the double-scaled Sachdev--Ye--Kitaev (SYK) model, Okuyama defined certain {\em discrete Weil--Petersson volumes} $N^q_{g,n}(b_1, \ldots, b_n)$, where $g$ is a non-negative integer, $b_1, \ldots, b_n$ are positive integers, and $q = e^{-\lambda}$ for $\lambda$ a natural parameter in the model~\cite{oku23}. Okuyama's paper includes some illustrative examples, such as the following.
\[
N^q_{1,2}(b_1, b_2) = 
\begin{cases}
\frac{1}{384} (b_1^4 + b_2^4) + \frac{1}{192} b_1^2b_2^2 + \frac{8 \zq(2)-1}{32} (b_1^2 + b_2^2) + \frac{5 \zq(4)}{2} + \frac{7 \zq(2)^2}{2} - \frac{\zq(2)}{2} + \frac{1}{12}, & \!\!\!\text{for } b_1, b_2 \text{ even}, \\
\frac{1}{384} (b_1^4 + b_2^4) + \frac{1}{192} b_1^2b_2^2 + \frac{8 \zq(2)-1}{32} (b_1^2 + b_2^2) + \frac{5 \zq(4)}{2} + \frac{7 \zq(2)^2}{2} - \frac{\zq(2)}{2} + \frac{5}{96}, & \!\!\!\text{for } b_1, b_2 \text{ odd}, \\
0, & \!\!\!\text{otherwise}.
\end{cases}
\]
Here, $\zq(s)$ represents the following $q$-analogue of the Riemann zeta function.
\begin{equation} \label{eq:zetaq}
\zq(s) = \sum_{m=1}^\infty \frac{q^{ms/2}}{(1 - q^m)^s}
\end{equation}

Okuyama verified in the cases $(g,n) = (0,3), (1,1), (1,2)$ and $(2,1)$ that $N^q_{g,n}(b_1, \ldots, b_n)$ is a quasi-polynomial in $b_1^2, \ldots, b_n^2$ with coefficients that are polynomial in the even $q$-zeta values $\zq(2), \zq(4), \zq(6), \ldots$. Here and throughout, we use the term {\em quasi-polynomial} to refer to a function on tuples of positive integers that is polynomial when restricted to residue classes modulo 2. Okuyama furthermore verified in these cases that a certain $q \to 1$ limit of $N^q_{g,n}(b_1, \ldots, b_n)$ recovers the Weil--Petersson volume $V_{g,n}(L_1, \ldots, L_n)$ of the moduli space ${\mathcal M}_{g,n}(L_1, \ldots, L_n)$ of hyperbolic surfaces with genus $g$ and $n$ labelled geodesic boundaries of lengths $L_1, \ldots, L_n$. These were famously studied by Mirzakhani, who showed that they are polynomial in $L_1^2, \ldots, L_n^2$ and that the coefficients are intersection numbers on the Deligne--Mumford compactification $\overline{\mathcal M}_{g,n}$ of the moduli space of curves~\cite{mir07b}.

In the present work, we provide mathematically rigorous proofs of the two aforementioned phenomena, which we state as the following theorems. Define the degree map $\deg:\mathbb{Q}[b_1, b_2, \ldots, \zq(2), \zq(4), \zq(6), \ldots] \to \bn$
by assigning
\[
\deg b_i = 1 \qquad \text{and} \qquad \deg \zeta_q(2k) = 2k,
\]
and extending in the usual way to monomials using additivity across products and then to polynomials by taking a maximum over monomials. This combines the usual notion of degree of a monomial in $b_1, b_2, \ldots$ with a notion of degree for its coefficient as an element of $\mathbb{Q}[\zq(2), \zq(4), \zq(6), \ldots]$ to form a graded ring. One can furthermore extend this definition of degree to a quasi-polynomial in $b_1, b_2, \ldots$ by considering it as a collection of polynomials and taking the maximal degree of these.

\begin{theorem} \label{thm:qzeta}
For $(g,n) \neq (0,1)$ or $(0,2)$, Okuyama's discrete volume $N^q_{g,n}(b_1, \ldots, b_n)$ is a quasi-polynomial in $b_1^2, \ldots, b_n^2$ with coefficients in $\mathbb{Q}[\zq(2), \zq(4), \zq(6), \ldots]$ of degree $6g-6+2n$, using the degree defined above.
\end{theorem}

\begin{theorem} \label{thm:WPvolumes}
For $(g,n) \neq (0,1)$ or $(0,2)$, Okuyama's discrete volume $N^q_{g,n}(b_1, \ldots, b_n)$ recovers the Weil--Petersson volume $V_{g,n}(L_1, \ldots, L_n)$ via the limit
\[
\lim_{\lambda \to 0} \lambda^{6g-6+2n} \, {N}^q_{g,n} \Big( \frac{L_1}{\lambda}, \ldots, \frac{L_n}{\lambda} \Big) = 2^{3-2g-n} \, V_{g,n}(L_1, \ldots, L_n).
\]
\end{theorem}

Okuyama's definition of the discrete Weil--Petersson volumes uses the topological recursion of Chekhov--Eynard--Orantin~\cite{che-eyn06,eyn-ora07a}. The topological recursion takes as input the data of a spectral curve and produces a family of correlation differentials $\omega_{g,n}$, for $g \geq 0$ and $n \geq 1$. For various choices of spectral curve, these correlation differentials store interesting enumerative information in their expansion coefficients. See \cref{subsec:TR} for further information on the topological recursion. More precisely, Okuyama defines the spectral curve
\[
x(z) = z + z^{-1} \qquad \text{and} \qquad y(z) = \frac{1}{2} (z - z^{-1}) \prod_{k=1}^\infty \frac{(1-q^kz^2) \, (1-q^kz^{-2})}{(1-q^k)^2},
\]
where $z \in \mathbb{CP}^1$, and expands the corresponding correlation differentials thus:
\begin{equation} \label{eq:omegaandN}
\omega_{g,n}(z_1, \ldots, z_n) = \sum_{b_1, \ldots, b_n = 1}^\infty N^q_{g,n}(b_1, \ldots, b_n) \prod_{i=1}^n b_i z_i^{b_i-1} \, \dd z_i.
\end{equation}

Given this definition of the discrete Weil--Petersson volume $N^q_{g,n}(b_1, \ldots, b_n)$, it is natural that the proofs of \cref{thm:qzeta,thm:WPvolumes} fundamentally use the topological recursion. The proof of \cref{thm:qzeta} is via an analysis of the recursion kernel appearing in the topological recursion, while the proof of \cref{thm:WPvolumes} relates Okuyama's spectral curve with the spectral curve known to store the Weil--Petersson volumes, first stated by Eynard and Orantin~\cite{eyn-ora07b}.

A proof of \cref{thm:WPvolumes} appeared in the recent work of Giacchetto, Maity and Mazenc~\cite{gia-mai-maz25}. They analyse so-called pruned correlators of the Gaussian Unitary Ensemble (GUE) with even potential, of which Okuyama's discrete volumes constitute an example. Their approach is different from ours in that it uses the Kontsevich--Soibelman framework for topological recursion, in which the initial data are encapsulated in a quantum Airy structure~\cite{kon-soi18,ABCO24}.

In recent work of the authors, we defined a $q$-analogue of Mirzakhani's recursion for Weil--Petersson volumes~\cite{do-nor25}. This recursion produces $q$-deformations of the Weil--Petersson volumes, given by polynomials $V^q_{g,n}(L_1, \ldots, L_n) \in \bq[[q]] [L_1^2, \ldots, L_n^2]$. It was conjectured that these polynomials agree with the top degree part of Okuyama's quasi-polynomials and we prove this here.

\begin{theorem} \label{thm:topdegterms}
For $(g,n) \neq (0,1)$ or $(0,2)$, Okuyama's discrete volume $N^q_{g,n}(b_1, \ldots, b_n)$ satisfifes
\[
{N}^q_{g,n} (b_1, \ldots, b_n) = 2^{3-2g-n} \, V^q_{g,n}(b_1, \ldots, b_n ) + [\,\text{lower degree terms}\,].
\]
\end{theorem}

\Cref{thm:topdegterms} immediately implies that the top degree terms of Okuyama's quasi-polynomials are in fact polynomials. The proof of \cref{thm:topdegterms} involves an analysis of Okuyama's spectral curve, which leads to a spectral curve that produces the polynomials $V^q_{g,n}(L_1, \ldots, L_n) \in \bq[[q]] [L_1^2, \ldots, L_n^2]$ via topological recursion. The relationship between topological recursion and tautological intersection numbers on moduli spaces of stable curves then allows us to define $q$-deformations of the classical Weil--Petersson volumes that were not defined previously. This produces a sequence $V_2(q), V_3(q), V_4(q), \ldots \in\bq[[q]]$ with the property that $\displaystyle\lim_{q \to 1} (1-q)^{6g-6} \, V_g(q) = \mathrm{Vol}^{\mathrm{WP}}(\cm_g)$, the classical Weil--Petersson volume of the moduli space of genus $g$ curves. For the simplest case $g=2$, we have
\[
V_2(q) = \frac{191}{90} \zeta_q(2)^3 + \frac{13}{3} \zeta_q(2) \zeta_q(4) + \frac{35}{18} \zeta_q(6),
\]
which satisfies
\[
\lim_{q \to 1} (1-q)^6 \, V_2(q) = \frac{191}{90} \zeta(2)^3 + \frac{13}{3} \zeta(2) \zeta(4) + \frac{35}{18} \zeta(6) = \frac{43\pi^6}{2160} = \mathrm{Vol}^{\mathrm{WP}}(\cm_2).
\]

A particularly appealing aspect of Okuyama's discrete volumes is that they store: the geometric information of Weil--Petersson volumes through the $q \to 1$ limit of \cref{thm:WPvolumes}~\cite{mir07a,mir07b,eyn-ora07b}; the combinatorial information of enumerating lattice points in moduli spaces of curves in the $q \to 0$ limit~\cite{nor10,nor13}; and the algebro-geometric information of psi-class intersection numbers on moduli spaces of curves in the coefficients of its leading order terms~\cite{eyn-ora07a,kon92,wit91}.

Okuyama's empirical observations demonstrate a direct link between DS-SYK correlators and Weil--Petersson volumes that we prove in the present work. Such a link hints at a ``quantum'' Weil--Petersson geometry, which should underpin the $q$-analogue of the Weil--Petersson volumes recently introduced by the authors~\cite{do-nor25}. It also hints at a combinatorial-geometric approach to double-scaled SYK correlators, which may arise through interpreting them in the context of map enumeration~\cite{do-tay25}.

In \cref{sec:background}, we provide minimal introductions to topological recursion, the double-scaled SYK model, Okuyama's spectral curve, and Weil--Petersson volumes. Readers familiar with any of these topics are invited to skip the corresponding subsections. In \cref{sec:proofs}, we present the proofs of \cref{thm:qzeta,thm:WPvolumes}, and in \cref{sec:top}, we present the proof of \cref{thm:topdegterms}.

\section{Background} \label{sec:background}

\subsection{Topological recursion} \label{subsec:TR}

The topological recursion of Chekhov--Eynard--Orantin formalises and generalises the notion of loop equations from the theory of matrix models~\cite{che-eyn06,eyn-ora07a}. In its original formulation, the topological recursion takes as input a {\em spectral curve} $(\C, x, y, B)$ consisting of a compact Riemann surface $\C$, a symmetric bidifferential $B$ on $\C \times \C$, and two meromorphic functions $x, y: \C \to \mathbb{C}$, which are required to satisfy some technical assumptions. From this data, the topological recursion produces a collection of {\em correlation differentials} $\omega_{g,n}(p_1, \ldots, p_n)$ for $p_i \in \C$, where $g \geq 0$ and $n \geq 1$ are integers.

\begin{itemize}
\item {\em Base cases.} The {\em unstable} correlation differentials are defined from the spectral curve as follows.
\[
\omega_{0,1}(p_1) = y(p_1) \, \dd x(p_1) \qquad \qquad \omega_{0,2}(p_1, p_2) = B(p_1, p_2)
\]

\item {\em Recursion kernel.} The {\em branch points} of the spectral curve are the points at which $\dd x = 0$ and we assume that the order of vanishing is one. Thus, at a branch point $\alpha$, there is a locally defined non-trivial holomorphic involution $\sigma_\alpha$ such that $x(p) = x(\sigma_\alpha(p))$ for all points $p$ in a neighbourhood of $\alpha \in \C$. Define the {\em recursion kernel} in a neighbourhood of a branch point $\alpha \in \C$ to be
\begin{equation} \label{eq:recursionkernel}
K_\alpha(p_1, p) = \frac{1}{2} \frac{\int_p^{\sigma_\alpha(p)} \omega_{0,2}(p_1, p')}{\omega_{0,1}(p) - \omega_{0,1}(\sigma_\alpha(p))}.
\end{equation}

\item {\em Recursion.} Every other $\omega_{g,n}(p_1, \ldots, p_n)$ is known as a {\em stable} correlation differential and satisfies the following recursion. Here, we use the notation $S = \{2, 3, \ldots, n\}$ and write $\pp_I = (p_{i_1}, p_{i_2}, \ldots, p_{i_m})$ for $I = \{i_1, i_2, \ldots, i_m\}$.
\begin{align} \label{eq:TR}
\omega_{g,n}(p_1, \pp_S) = \sum_{\alpha} \mathop{\mathrm{Res}}_{p=\alpha} K(p_1, p) \Bigg[ &\omega_{g-1,n+1}(p, \sigma_\alpha(p), \pp_S) \notag \\
&+ \sum_{\substack{g_1+g_2=g \\ I \sqcup J = S}}^\circ \omega_{g_1,|I|+1}(p, \pp_I) \, \omega_{g_2,|J|+1}(\sigma_\alpha(p), \pp_J) \Bigg]
\end{align}
The sum is over the branch points $\alpha$ of the spectral curve and the $\circ$ over the inner summation indicates that we exclude terms that involve $\omega_{0,1}$.
\end{itemize}

In the present work, we only consider {\em rational spectral curves}, in which the underlying compact Riemann surface is $\mathcal{C} = \mathbb{CP}^1$. This leads to the symmetric bidifferential being given by
\begin{equation} \label{eq:bergman}
B(z_1, z_2) = \frac{\dd z_1 \, \dd z_2}{(z_1-z_2)^2},
\end{equation}
where we take $z_i$ to be the usual coordinate on $\mathbb{CP}^1$. Thus, a rational spectral curve is essentially specified by the two meromorphic functions $x, y: \mathbb{CP}^1 \to \mathbb{C}$. The topological recursion produces correlation differentials that satisfy many interesting properties, but we state only the most fundamental ones in the following.

\begin{proposition} \label{prop:TRproperties}
The stable correlation differential $\omega_{g,n}$ is a symmetric meromorphic multidifferential on $\C$ with poles only at the branch points. At each branch point, the order of the pole is at most $6g-4+2n$ and the residue is equal to 0.
\end{proposition}

One can consult the seminal papers of Eynard and Orantin for further details~\cite{eyn-ora07a,eyn-ora09}, keeping in mind that the notion of a spectral curve has been relaxed in subsequent works, while the technical assumptions have become more refined. In particular, a spectral curve need not be compact and the order of vanishing of $\dd x$ at the branch points may be arbitrary. Furthermore, there is a more algebraic approach to topological recursion by Kontsevich and Soibelman, which uses as input the notion of an Airy structure~\cite{kon-soi18}.

Since its inception, topological recursion has found widespread application to topics beyond matrix models such as: map enumeration~\cite{do-man14,DOPS18,kaz-zog15}, Weil--Petersson volumes of moduli spaces~\cite{eyn-ora07b}, Hurwitz numbers and their variations~\cite{ACEH20, bou-mar08,BHLM14,do-dye-mat17,do-lei-nor16,eyn-mul-saf11,BDKLM23}, Gromov--Witen theory of the sphere~\cite{nor-sco14}, topological string theory~\cite{BKMP09,eyn-ora15,fan-liu-zon20}, cohomological field theory~\cite{DOSS14}, free probability~\cite{BCGLS21}, Jackiw--Teitelboim gravity~\cite{EGGLS24}, and conjecturally quantum knot invariants~\cite{bor-eyn15,dij-fuj-man11,GJKS15}.

\subsection{The double-scaled SYK model}

The SYK model has attracted attention as an exactly solvable toy model for various physical phenomena, such as quantum chaos, strange metals and holography~\cite{sac-ye93,kit15}. It is a quantum mechanical model that describes the $p$-local interaction of $N$ Majorana fermions. These Majorana fermions $\psi_1, \psi_2, \ldots, \psi_N$ obey the anti-commutation relation $\{ \psi_i, \psi_j \} = 2 \delta_{i,j}$ and naturally act on the vector space with basis vectors $\psi_{i_1} \psi_{i_2} \cdots \psi_{i_m} \ket{0}$, where $1 \leq i_1 < i_2 < \cdots < i_m \leq N$ and $\ket{0}$ denotes the vacuum vector. The Hamiltonian for the SYK model is given by the following expression, where $1 \leq p \leq N$ is assumed to be even.
\[
H = i^{p/2} \sum_{1 \leq i_1 < i_2 < \cdots < i_p \leq N} J_{i_1 i_2 \cdots i_p} \, \psi_{i_1} \psi_{i_2} \cdots \psi_{i_p}
\]
Here, $\{ J_{i_1 i_2 \cdots i_p} \mid 1 \leq i_1 < i_2 < \cdots < i_p \leq N \}$ are random couplings drawn from independent Gaussian ensembles with mean 0 and variance $\binom{N}{p}^{-1}$. Thus, if we use $\langle \, \cdot \, \rangle_J$ to denote the average over the Gaussians and $I, I_1, I_2$ denote cardinality $p$ subsets of $\{1, 2, \ldots, N\}$, then
\begin{equation} \label{eq:Javerages}
\langle J_I \rangle_J = 0 \qquad \text{and} \qquad \langle J_{I_1} J_{I_2} \rangle_J = \delta_{I_1,I_2} \binom{N}{p}^{-1}.
\end{equation}

There has been a great deal of recent interest in the double-scaled Sachdev-Ye-Kitaev (DS-SYK) model, in which the parameters $N$ and $p$ approach infinity with $\lambda = \frac{2p^2}{N}$ approaching a fixed limit~\cite{cot-et-al17}. As usual, one is interested in computing the moments, the correlators and the partition function for the DS-SYK model. An initial observation is that the odd moments vanish since the expectation of a product of an odd number of Gaussians with mean zero vanishes. For the even moments, we have the following elegant result established by invoking the Wick--Isserlis theorem, where we introduce the parameter
\[
q = e^{-\lambda} = e^{-2p^2/N}.
\]

\begin{proposition} [Berkooz, Isachenkov, Narovlansky and Torrents~\cite{BINT19}] \label{prop:moments}
The even moments of the DS-SYK model in the large $N$ limit are given by
\[
\langle \, \tr H^{2k} \, \rangle_J = \sum_{k\text{-chord diagrams}} q^{\# \text{intersections}}.
\]
A $k$-chord diagram is a pairing of $2k$ points around a circle using $k$ chords. An intersection is produced by two chords whose endpoints alternate around the circle. For example, $\langle \, \tr H^6 \, \rangle_J = q^3 + 3q^2 + 6q + 5$, with contributions coming from the fifteen 3-chord diagrams, as shown in \cref{fig:chorddiagrams}.
\end{proposition}

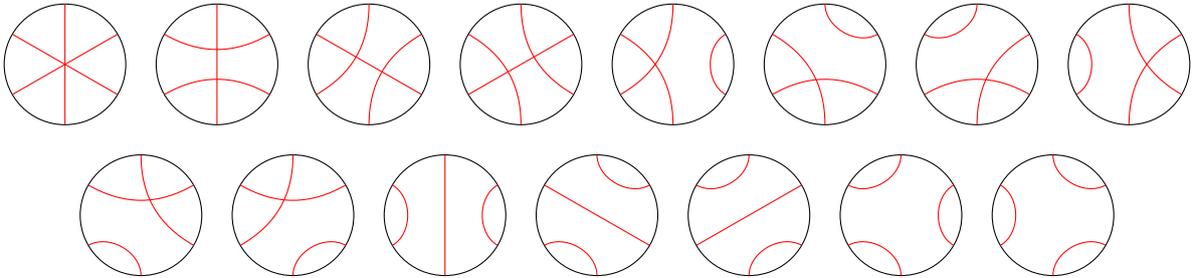
\begin{figure}[ht!]
\centering
\begin{tikzpicture}
\def\r{0.8}
\def\x{2}
\coordinate (O) at (0,0);
\coordinate (p1) at ({\r*cos(90)},{\r*sin(90)});
\coordinate (p2) at ({\r*cos(150)},{\r*sin(150)});
\coordinate (p3) at ({\r*cos(210)},{\r*sin(210)});
\coordinate (p4) at ({\r*cos(270)},{\r*sin(270)});
\coordinate (p5) at ({\r*cos(330)},{\r*sin(330)});
\coordinate (p6) at ({\r*cos(30)},{\r*sin(30)});

\begin{scope} [every coordinate/.style={shift={(-3.5*\x,0)}}]
\draw ([c]O) circle (\r);
\draw[red] ([c]p1) to[out=270,in=90] ([c]p4);
\draw[red] ([c]p2) to[out=330,in=150] ([c]p5);
\draw[red] ([c]p3) to[out=30,in=210] ([c]p6);
\end{scope}

\begin{scope} [every coordinate/.style={shift={(-2.5*\x,0)}}]
\draw ([c]O) circle (\r);
\draw[red] ([c]p1) to[out=270,in=90] ([c]p4);
\draw[red] ([c]p2) to[out=330,in=210] ([c]p6);
\draw[red] ([c]p3) to[out=30,in=150] ([c]p5);
\end{scope}

\begin{scope} [every coordinate/.style={shift={(-1.5*\x,0)}}]
\draw ([c]O) circle (\r);
\draw[red] ([c]p1) to[out=270,in=30] ([c]p3);
\draw[red] ([c]p2) to[out=330,in=150] ([c]p5);
\draw[red] ([c]p4) to[out=90,in=210] ([c]p6);
\end{scope}

\begin{scope} [every coordinate/.style={shift={(-0.5*\x,0)}}]
\draw ([c]O) circle (\r);
\draw[red] ([c]p1) to[out=270,in=150] ([c]p5);
\draw[red] ([c]p2) to[out=330,in=90] ([c]p4);
\draw[red] ([c]p3) to[out=30,in=210] ([c]p6);
\end{scope}

\begin{scope} [every coordinate/.style={shift={(0.5*\x,0)}}]
\draw ([c]O) circle (\r);
\draw[red] ([c]p1) to[out=270,in=30] ([c]p3);
\draw[red] ([c]p2) to[out=330,in=90] ([c]p4);
\draw[red] ([c]p5) to[out=150,in=210] ([c]p6);
\end{scope}

\begin{scope} [every coordinate/.style={shift={(1.5*\x,0)}}]
\draw ([c]O) circle (\r);
\draw[red] ([c]p1) to[out=270,in=210] ([c]p6);
\draw[red] ([c]p2) to[out=330,in=90] ([c]p4);
\draw[red] ([c]p3) to[out=30,in=150] ([c]p5);
\end{scope}

\begin{scope} [every coordinate/.style={shift={(2.5*\x,0)}}]
\draw ([c]O) circle (\r);
\draw[red] ([c]p1) to[out=270,in=330] ([c]p2);
\draw[red] ([c]p3) to[out=30,in=150] ([c]p5);
\draw[red] ([c]p4) to[out=90,in=210] ([c]p6);
\end{scope}

\begin{scope} [every coordinate/.style={shift={(3.5*\x,0)}}]
\draw ([c]O) circle (\r);
\draw[red] ([c]p1) to[out=270,in=150] ([c]p5);
\draw[red] ([c]p2) to[out=330,in=30] ([c]p3);
\draw[red] ([c]p4) to[out=90,in=210] ([c]p6);
\end{scope}

\begin{scope} [every coordinate/.style={shift={(-3*\x,-\x)}}]
\draw ([c]O) circle (\r);
\draw[red] ([c]p1) to[out=270,in=150] ([c]p5);
\draw[red] ([c]p2) to[out=330,in=210] ([c]p6);
\draw[red] ([c]p3) to[out=30,in=90] ([c]p4);
\end{scope}

\begin{scope} [every coordinate/.style={shift={(-2*\x,-\x)}}]
\draw ([c]O) circle (\r);
\draw[red] ([c]p1) to[out=270,in=30] ([c]p3);
\draw[red] ([c]p2) to[out=330,in=210] ([c]p6);
\draw[red] ([c]p4) to[out=90,in=150] ([c]p5);
\end{scope}

\begin{scope} [every coordinate/.style={shift={(-1*\x,-\x)}}]
\draw ([c]O) circle (\r);
\draw[red] ([c]p1) to[out=270,in=90] ([c]p4);
\draw[red] ([c]p2) to[out=330,in=30] ([c]p3);
\draw[red] ([c]p5) to[out=150,in=210] ([c]p6);
\end{scope}

\begin{scope} [every coordinate/.style={shift={(0*\x,-\x)}}]
\draw ([c]O) circle (\r);
\draw[red] ([c]p1) to[out=270,in=210] ([c]p6);
\draw[red] ([c]p2) to[out=330,in=150] ([c]p5);
\draw[red] ([c]p3) to[out=30,in=90] ([c]p4);
\end{scope}

\begin{scope} [every coordinate/.style={shift={(1*\x,-\x)}}]
\draw ([c]O) circle (\r);
\draw[red] ([c]p1) to[out=270,in=330] ([c]p2);
\draw[red] ([c]p3) to[out=30,in=210] ([c]p6);
\draw[red] ([c]p4) to[out=90,in=150] ([c]p5);
\end{scope}

\begin{scope} [every coordinate/.style={shift={(2*\x,-\x)}}]
\draw ([c]O) circle (\r);
\draw[red] ([c]p1) to[out=270,in=330] ([c]p2);
\draw[red] ([c]p3) to[out=30,in=90] ([c]p4);
\draw[red] ([c]p5) to[out=150,in=210] ([c]p6);
\end{scope}

\begin{scope} [every coordinate/.style={shift={(3*\x,-\x)}}]
\draw ([c]O) circle (\r);
\draw[red] ([c]p1) to[out=270,in=210] ([c]p6);
\draw[red] ([c]p2) to[out=330,in=30] ([c]p3);
\draw[red] ([c]p4) to[out=90,in=150] ([c]p5);
\end{scope}
\end{tikzpicture}
\caption{The fifteen 3-chord diagrams, which contribute to the calculation of $\langle \, \tr H^6 \, \rangle_J = q^3 + 3q^2 + 6q + 5$.}
\label{fig:chorddiagrams}
\end{figure}

\cref{prop:moments} suggests that there may be interesting combinatorics underlying DS-SYK correlators that may propagate to their large $N$ expansions. However, such combinatorics is yet to be well understood~\cite{BBNR21}.

\subsection{Okuyama's spectral curve}

Jafferis, Kolchmeyer, Mukhametzhanov and Sonner introduced an $N \times N$ Hermitian one-matrix model whose moments match those of the DS-SYK model described by \cref{prop:moments}~\cite{JKMS23}. They furthermore argued that the connected correlators for these two models agree to leading order -- that is, at genus zero in the large $N$ expansion. The potential for the matrix model has the following explicit form, where $T_{2n}$ denotes the Chebyshev polynomial of the first kind.
\[
V(x) = \sum_{n=1}^\infty \frac{(-1)^{n+1}}{n} \big( q^{n(n+1)/2} + q^{n(n-1)/2} \big) \, T_{2n} \big( \tfrac{\sqrt{1-q}}{2} x \big)
\]

Hermitian one-matrix models are known to be governed by the topological recursion when the potential is polynomial and conjecturally in many other cases~\cite{eyn-ora09}. Thus, Okuyama was motivated to introduce the following rational spectral curve\footnote{The spectral curve presented here is actually a rescaled version of that introduced by Okuyama, which takes the form $\widetilde{x}(z) = \frac{1}{\sqrt{1-q}} \, x(z)$ and $\widetilde{y}(z) = \sqrt{1-q} \, (q;q)_\infty^3 \, y(z)$. Our rescaling has little bearing on the mathematics and is mainly for notational convenience. Furthermore, observe that $y(z)$ here is not a rational function -- we will have more to say on this in \cref{subsec:qzeta}.} to compute the correlators of the matrix model~\cite{oku23}.
\begin{equation} \label{eq:okuyama}
x(z) = z+z^{-1} \qquad \qquad y(z) = \frac{1}{2}(z-z^{-1}) \prod_{k=1}^\infty \frac{(1-q^kz^2) \, (1-q^kz^{-2})}{(1-q^k)^2}
\end{equation}

Okuyama applies the topological recursion to this spectral curve and expands at $z_i = 0$ to obtain
\[
\omega_{g,n}(z_1, \ldots, z_n) = \sum_{b_1, \ldots, b_n=1}^\infty N^q_{g,n}(b_1, \ldots, b_n) \prod_{i=1}^n b_i z_i^{b_i-1} \, \dd z_i.
\]
The quantities $N^q_{g,n}(b_1, \ldots, b_n)$ -- which we will refer to as {\em Okuyama's discrete volumes} -- enter into the large~$N$ expansion of connected correlators for the matrix model via the following formulas, where $(q; q)_\infty$ is a $q$-Pochhammer symbol and $I_\alpha$ denotes the modified Bessel function of the first kind.
\begin{align*}
\big\langle \tr e^{\beta_1M} \cdots \tr e^{\beta_nM} \big\rangle^c &= \sum_{g=0}^\infty \frac{1}{N^{2g-2+n}} Z_{g,n}(\beta_1, \ldots, \beta_n) \\
Z_{g,n}(\beta_1, \ldots, \beta_n) &= \frac{1}{(q;q)_\infty^{6g-6+3n}} \sum_{b_1, \ldots, b_n=1}^\infty N^q_{g,n}(b_1, \ldots, b_n) \prod_{i=1}^n b_i \, I_{b_i} \big( \tfrac{2\beta_i}{\sqrt{1-q}} \big)
\end{align*}

The analogous DS-SYK connected correlators also admit a large $N$ expansion of the following form.
\[
\big\langle \tr e^{-\beta_1H} \cdots \tr e^{-\beta_nH} \big\rangle_J^c = \sum_{g=0}^\infty \frac{1}{N^{2g-2+n}} \widetilde{Z}_{g,n}(\beta_1, \ldots, \beta_n)
\]
Okuyama states that although ``there is no obvious reason to expect the agreement\ldots~it is tempting to speculate that these two computations actually agree at all genera''~\cite{oku23} -- in other words, that for all $g \geq 0$, we have $\widetilde{Z}_{g,n}(\beta_1, \ldots, \beta_n) = Z_{g,n}(\beta_1, \ldots, \beta_n)$. The universality and naturality of the topological recursion could be considered as evidence towards this claim.

The spectral curve of \cref{eq:okuyama} possesses the obvious symmetries $x(z) = x(z^{-1}) = -x(-z)$ and $y(z) = -y(z^{-1}) = -y(-z)$. This leads to the following symmetry for the stable correlation differentials.
\begin{equation} \label{eq:omegasymmetries}
\omega_{g,n}(z_1, \ldots, z_i, \ldots, z_n) = - \omega_{g,n}(z_1, \ldots, z_i^{-1}, \ldots, z_n)
\end{equation}
We furthermore have the following result, which continues to apply to the spectral curve of \cref{eq:okuyama}, despite $y(z)$ not being rational.

\begin{theorem} [Norbury and Scott~\cite{nor-sco13}] \label{thm:norsco}
For a rational spectral curve with $x(z) = \alpha + \gamma(z + z^{-1})$ and $y(z)$ a rational function with $y'(1) \neq 0$ and $y'(-1) \neq 0$, the stable correlation differentials satisfy
\[
\omega_{g,n}(z_1, \ldots, z_n) = \sum_{b_1, \ldots, b_n=1}^\infty N_{g,n}(b_1, \ldots, b_n) \prod_{i=1}^n b_i z_i^{b_i-1} \, \dd z_i,
\]
where $N_{g,n}(b_1, \ldots, b_n)$ is a symmetric quasi-polynomial in $b_1^2, \ldots, b_n^2$ of degree $3g-3+n$.
\end{theorem}

\subsection{Weil--Petersson volumes}

The moduli space ${\mathcal M}_{g,n}(L_1, \ldots, L_n)$ of hyperbolic surfaces with $n$ geodesic boundaries of lengths $L_1, \ldots, L_n$ carries a symplectic structure through the Weil--Petersson symplectic form $\omega_{WP}$. Mirzakhani showed that the corresponding Weil--Petersson volume
\[
V_{g,n}(L_1, \ldots, L_n) = \int_{{\mathcal M}_{g,n}(L_1, \ldots, L_n)} \frac{\omega_{WP}^{3g-3+n}}{(3g-3+n)!}
\]
is a polynomial in $L_1^2, \ldots, L_n^2$ of degree $3g-3+n$ and that the coefficients are intersection numbers on the Deligne--Mumford compactification $\overline{\mathcal M}_{g,n}$ of the moduli space of curves~\cite{mir07b}. Furthermore, Mirzakhani provided the following recursion from which one can explicitly calculate all such Weil--Petersson volumes~\cite{mir07a}.

\begin{itemize}
\item {\em Base cases.} The Weil--Petersson volumes for $(g,n) = (0,1), (0,2), (0,3)$ and $(1,1)$ are as follows.
\[
V_{0,1}(L_1) = 0 \qquad V_{0,2}(L_1, L_2) = 0 \qquad V_{0,3}(L_1, L_2, L_3) = 1 \qquad V_{1,1}(L_1) = \frac{1}{48} L_1^2 + \frac{\pi^2}{12}
\]

\item {\em Recursion.} Every other Weil--Petersson volume $V_{g,n}(L_1, \ldots, L_n)$ satisfies the following recursion. Here, we use the notation $S = \{2, 3, \ldots, n\}$ and write $\LL_I = (L_{i_1}, L_{i_2}, \ldots, L_{i_m})$ for $I = \{i_1, i_2, \ldots, i_m\}$. We also define the recursion kernel $H(x, y) = (1 + \exp \frac{x+y}{2})^{-1} + (1 + \exp \frac{x-y}{2})^{-1}$.
\begin{align} \label{eq:mirzakhani}
\frac{\partial}{\partial L_1} L_1 V_{g,n}(L_1, \LL_S) &= \frac{1}{2} \int_0^\infty \!\!\! \int_0^\infty xy \, H(x+y, L_1) \, V_{g-1,n+1}(x, y, \LL_S) \, \dd x \, \dd y \notag \\
&+ \frac{1}{2} \sum_{\substack{g_1+g_2=g \\ I \sqcup J = S}} \int_0^\infty \!\!\! \int_0^\infty xy \, H(x+y,L_1) \, V_{g_1,|I|+1}(x, \LL_I) \, V_{g_2,|J|+1}(y, \LL_J) \, \dd x \ \dd y \notag \\
&+ \frac{1}{2} \sum_{k=2}^n \int_0^\infty x \big[ H(x, L_1+L_k) + H(x, L_1-L_k) \big] \, V_{g,n-1}(x, \LL_{S \setminus \{k\}}) \, \dd x
\end{align}

\item {\em Evaluation.} The integrals appearing in the recursion above can be evaluated using the formulas
\begin{align*}
\int_0^\infty x^{2k+1} \, H(x,t) \, \dd x &= F_{2k+1}(t) \\
\int_0^\infty \!\!\! \int_0^\infty x^{2a+1} y^{2b+1} \, H(x+y, t) \, \dd x \, \dd y &= \frac{(2a+1)! \, (2b+1)!}{(2a+2b+3)!} \, F_{2a+2b+3}(t),
\end{align*}
where $F_{2k+1}(t) = (2k+1)! \, \displaystyle\sum_{i=0}^{k+1} (2^{2i+1}-4) \, \zeta(2i) \, \frac{t^{2k+2-2i}}{(2k+2-2i)!}$.
\end{itemize}

\cref{app:WPvolumes} contains a table of Weil--Petersson volumes $V_{g,n}(L_1, \ldots, L_n)$ for some small values of $g$ and~$n$. For an introduction to Weil--Petersson volumes, particularly with regards to Mirzakhani's recursion, see the survey article from the Handbook of Moduli~\cite{do13}.

Mirzakhani's recursion in \cref{eq:mirzakhani} bears a strong resemblance to the topological recursion in \cref{eq:TR}, with the respective terms being essentially in one-to-one correspondence. Thus, it is natural to surmise that Mirzakhani's recursion is an instance of the topological recursion on a particular spectral curve and this is indeed the case.

\begin{theorem}[Eynard and Orantin~\cite{eyn-ora07b}]
Topological recursion applied to the rational spectral curve\footnote{Although $y(t)$ is not a rational function, one can make sense of the topological recursion by considering the sequence of rational functions $y^M(t) = \displaystyle\sum_{k=0}^M \frac{(-1)^k (2\pi)^{2k}}{(2k+1)!} t^{2k+1}$. For fixed $(g,n)$, the correlation differential $\omega_{g,n}$ stabilises at finite $M$.}
\begin{equation} \label{eq:WPvolumes}
x(t) = \frac{t^2}{2} \qquad \text{and} \qquad y(t) = \frac{1}{2\pi} \sin(2\pi t)
\end{equation}
produces stable correlation differentials that satisfy
\[
\frac{\omega^{\mathrm{WP}}_{g,n}(t_1, \ldots, t_n)}{\dd t_1 \cdots \dd t_n} = \int_0^\infty \!\!\! \cdots \! \int_0^\infty V_{g,n}(L_1, \ldots, L_n) \prod_{i=1}^n L_i e^{-t_i L_i} \, \dd L_i.
\]
\end{theorem}

\section{Discrete volumes} \label{sec:proofs}

In this section, we present the proofs of \cref{thm:qzeta} concerning the appearance of $q$-zeta values in Okuyama's discrete volumes and \cref{thm:WPvolumes} regarding the Weil--Petersson volumes arising in the $q \to 1$ limit.

\subsection{The appearance of \texorpdfstring{$q$}{q}-zeta values} \label{subsec:qzeta}

In its original form described in \cref{subsec:TR}, the topological recursion requires two meromorphic functions $x, y: \mathcal{C} \to \mathbb{C}$ as part of the input data~\cite{che-eyn06,eyn-ora07a}. Observe however that the function $y$ appearing in Okuyama's spectral curve -- see \cref{eq:okuyama} -- is not rational. One can still make sense of the situation in various ways. We propose to understand Okuyama's spectral curve as the limit of the following sequence of rational spectral curves indexed by a positive integer $M$.
\begin{equation} \label{eq:Mspectralcurve}
x(z) = z+z^{-1} \qquad \qquad y^M(z) = \frac{1}{2} (z-z^{-1}) \prod_{k=1}^M \frac{(1-q^kz^2) \, (1-q^kz^{-2})}{(1-q^k)^2}
\end{equation}
We will subsequently use the superscript $M$ to indicate quantities derived from this rational spectral curve, such as writing the corresponding stable correlation differentials as
\begin{equation} \label{eq:omegaMexpansion}
\omega_{g,n}^M(z_1, \ldots, z_n) = \sum_{b_1, \ldots, b_n=1}^\infty N^M_{g,n}(b_1, \ldots, b_n) \prod_{i=1}^n b_i z_i^{b_i-1} \, \dd z_i.
\end{equation}

Truncating the product appearing in $y(z)$ leads to expressions that involve the truncated $q$-zeta function
\begin{equation} \label{eq:zetaqN}
\zqm(s) = \sum_{m=1}^M \frac{q^{ms/2}}{(1 - q^m)^s}.
\end{equation}
Thus, \cref{thm:qzeta} is a direct consequence of the following result.

\begin{theorem} \label{thm:Mqzeta}
Let $M$ be a positive integer and assume that $(g,n) \neq (0,1)$ or $(0,2)$.
\begin{enumerate}[label=(\alph*)]
\item Then $N^M_{g,n}(b_1, \ldots, b_n)$ is a quasi-polynomial in $b_1^2, \ldots, b_n^2$. 

\item Each underlying polynomial of $N^M_{g,n}(b_1, \ldots, b_n)$ is an element of the graded ring
\[
\mathbb{Q}[\zqm(2), \zqm(4), \zqm(6), \ldots] [b_1^2, \ldots, b_n^2]
\]
of degree at most $6g-6+2n$. Here, we define the degree of $b_i^2$ to be 2 and the degree of $\zqm(2k)$ to be $2k$.

\item For a fixed parity class of $(b_1, \ldots, b_n)$, the coefficient of $\zqm(2)^{m_1} \zqm(4)^{m_2} \zqm(6)^{m_3} \cdots b_1^{2d_1} \cdots b_n^{2d_n}$ in the underlying polynomial of $N^M_{g,n}(b_1, \ldots, b_n)$ is independent of $M$.
\end{enumerate}
\end{theorem}

The remainder of this section will be dedicated to proving \cref{thm:Mqzeta}.

Recall that the {\em principal part} of a meromorphic 1-form at a particular point is the Laurent series at that point with only the terms of negative degree retained. The following introduces the notation that we will use for principal parts, along with an expression for the principal part as a residue.
\[
\omega(z) = \sum_{k = -\infty}^\infty a_k (z-\alpha)^k \, \dd z \qquad \Rightarrow \qquad \mathcal{P}[\omega(z)]_{z=\alpha} = \sum_{k = -\infty}^{-1} a_k (z-\alpha)^k \, \dd z = \mathop{\mathrm{Res}}_{w=\alpha} \frac{\dd z}{z-w} \, \omega(w)
\]
Observe that a meromorphic 1-form on $\mathbb{CP}^1$ is the sum of the principal parts at its poles.

\begin{proposition} [Do, Leigh and Norbury~\cite{do-lei-nor16}] \label{prop:principalparts}
For a rational spectral curve, the topological recursion of \cref{eq:TR} expresses the correlation differential $\omega_{g,n}(z_1, \ldots, z_n)$ as the sum of its principal parts with respect to $z_1$ via the equation
\begin{multline*}
\omega_{g,n}(z_1, \zz_S) = \sum_{\alpha} \mathcal{P} \Bigg[ \frac{1}{\omega_{0,1}(z_1) - \omega_{0,1}(\sigma_\alpha(z_1))} \bigg( \omega_{g-1,n+1}(z_1, \sigma_\alpha(z_1), \zz_S) \\
+ \sum_{\substack{g_1+g_2=g \\ I \sqcup J = S}}^\circ \omega_{g_1,|I|+1}(z_1, \zz_I) \, \omega_{g_2,|J|+1}(\sigma_\alpha(z_1), \zz_J) \bigg) \Bigg]_{z_1=\alpha}.
\end{multline*}
\end{proposition}

In the case of the spectral curve of \cref{eq:Mspectralcurve}, the branch points are at $\alpha = \pm 1$ and the local involutions are given by $\sigma_{+1}(z) = \sigma_{-1}(z) = \frac{1}{z}$. Given \cref{prop:principalparts}, it is natural to consider the Laurent expansions at the points $z = \pm 1$ of the expression
\[
\frac{1}{\omega_{0,1}^M(z) - \omega_{0,1}^M(\sigma_\alpha(z))} = \frac{1}{2 \, y^M(z) \, x'(z)} \frac{1}{\dd z}.
\]

\begin{lemma} \label{lem:expansions}
For every positive integer $M$, we have the expansions
\begin{align*}
\frac{1}{2 \, y^M(z) \, x'(z)} &= \frac{z^3}{(z^2-1)^2} \, \exp \bigg[ \sum_{m=2}^\infty \sum_{k=1}^{\lfloor m/2 \rfloor} a_{mk} \, \zqm(2k) \, (z-1)^m \bigg], \qquad \text{for } |z-1|<1, \\
\frac{1}{2 \, y^M(z) \, x'(z)} &= \frac{z^3}{(z^2-1)^2} \, \exp \bigg[ \sum_{m=2}^\infty \sum_{k=1}^{\lfloor m/2 \rfloor} b_{mk} \, \zqm(2k) \, (z+1)^m \bigg], \qquad \text{for } |z+1|<1,
\end{align*}
where $a_{mk}$ is the coefficient of $t^{m-2k}$ in $\frac{1}{k} (2 - t + t^2 - t^3 + t^4 - \cdots)^{2k}$ and $b_{mk}$ is the coefficient of $t^{m-2k}$ in $\frac{1}{k} (2 + t + t^2 + t^3 + t^4 + \cdots)^{2k}$.
\end{lemma}

\begin{proof}
We compute as follows, directly using the definition of the spectral curve of \cref{eq:Mspectralcurve}.
\begin{align*}
\log \big( 2 \, y^M(z) \, x'(z) \big) &= \log \frac{(z^2-1)^2}{z^3} + \sum_{i=1}^M \log \frac{(1-q^iz^2) \, (1-q^iz^{-2})}{(1-q^i)^2} \\
&= \log \frac{(z^2-1)^2}{z^3} + \sum_{i=1}^M \log \bigg( 1 - \frac{q^i}{(1-q^i)^2} \, (z-z^{-1})^2 \bigg) \\
&= \log \frac{(z^2-1)^2}{z^3} - \sum_{i=1}^M \sum_{k=1}^\infty \frac{1}{k} \, \frac{q^{ik}}{(1-q^i)^{2k}} \, (z-z^{-1})^{2k} \\
&= \log \frac{(z^2-1)^2}{z^3} - \sum_{k=1}^\infty \frac{1}{k} \, \zqm(2k) \, (z-z^{-1})^{2k}
\end{align*}

Now use $t = z-1$ to write, for $|t| < 1$,
\[
z - z^{-1} = t+1 - \frac{1}{t+1} = t ( 2 - t + t^2 - t^3 + t^4 - \cdots).
\]
Substituting this into the expression above yields, for $|z-1| < 1$,
\begin{align*}
\log \big( 2 \, y^M(z) \, x'(z) \big) &= \log \frac{(z^2-1)^2}{z^3} - \sum_{k=1}^\infty \frac{1}{k} \zqm(2k) \, t^{2k} ( 2 - t + t^2 - t^3 + t^4 - \cdots)^{2k} \\
&= \log \frac{(z^2-1)^2}{z^3} - \sum_{m=2}^\infty \sum_{k=1}^{\lfloor m/2 \rfloor} a_{mk} \, \zqm(2k) \, t^m.
\end{align*}
The desired expansion at $z = 1$ then follows by negating both sides, applying the exponential, and using $t = z-1$. The desired expansion at $z = -1$ can be obtained by similar means.
\end{proof}

Introduce the vector space $\Lambda^M_d(z_1, \ldots, z_n)$, which contains the multidifferentials that can be expressed in the form
\[
\sum_{\substack{k_1, \ldots, k_n \geq 0 \\ k_1 + \cdots + k_n \leq d \\ a_1, \ldots, a_n = \pm 1}} C^{a_1, \ldots, a_n}_{k_1, \ldots, k_n} \big( \zqm(2), \zqm(4), \ldots \big) \, \frac{\dd z_1}{(z_1-a_1)^{k_1+2}} \cdots \frac{\dd z_n}{(z_n-a_n)^{k_n+2}},
\]
where $C^{a_1, \ldots, a_n}_{k_1, \ldots, k_n} \big( \zqm(2), \zqm(4), \ldots \big)$ is a polynomial of $\zeta$-degree at most $d-\sum k_i$. So $\Lambda^M_d(z_1, \ldots, z_n)$ is spanned by terms of the form
\[
C \big( \zqm(2), \zqm(4), \ldots \big) \, \frac{\dd z_1}{(z_1-a_1)^{k_1+2}} \cdots \frac{\dd z_n}{(z_n-a_n)^{k_n+2}}
\]
of total degree at most $d$, where we define the total degree of such a term to be $k_1 + \cdots + k_n$ plus the $\zeta$-degree of $C \big( \zqm(2), \zqm(4), \ldots \big)$.

\begin{proposition} \label{prop:preproof1}
The stable correlation differentials computed by topological recursion on the spectral curve of \cref{eq:Mspectralcurve} satisfy
\[
\omega^M_{g,n}(z_1, \ldots, z_n) \in \Lambda^M_{6g-6+2n}(z_1, \ldots, z_n).
\]
\end{proposition}

\begin{proof}
From \cref{lem:expansions}, there exist polynomials $A^+_{-2}, A^+_{-1}, A^+_0, A^+_1, \ldots$ and $A^-_{-2}, A^-_{-1}, A^-_0, A^-_1, \ldots$ such that
\begin{align*}
\frac{1}{2 \, y^M(z) \, x'(z)} &= \sum_{m=-2}^\infty A^+_m \big( \zqm(2), \zqm(4), \ldots \big) \, (z-1)^m \\
\frac{1}{2 \, y^M(z) \, x'(z)} &= \sum_{m=-2}^\infty A^-_m \big( \zqm(2), \zqm(4), \ldots \big) \, (z+1)^m.
\end{align*}
In fact, \cref{lem:expansions} is rather explicit and it follows that for $m \geq 2$, the $\zeta$-degree of $A^+_m$ and $A^-_m$ is even and at most $m+2$.

We now proceed to prove the proposition by induction on $2g-2+n$. The base cases can be calculated explicitly and are given by 
\begin{align}
\omega_{0,3}^M(z_1, z_2, z_3) &= \frac{1}{2} \frac{\dd z_1}{(z_1-1)^2} \frac{\dd z_2}{(z_2-1)^2} \frac{\dd z_3}{(z_3-1)^2} - \frac{1}{2} \frac{\dd z_1}{(z_1+1)^2} \frac{\dd z_2}{(z_2+1)^2} \frac{\dd z_3}{(z_3+1)^2} \notag \\
\omega_{1,1}^M(z_1) &= \frac{1}{16} \frac{\dd z_1}{(z_1-1)^4} + \frac{1}{16} \frac{\dd z_1}{(z_1-1)^3} + \left( \frac{1}{4} \zqm(2) - \frac{1}{32} \right) \frac{\dd z_1}{(z_1-1)^2} \notag \\
&- \frac{1}{16} \frac{\dd z_1}{(z_1+1)^4} + \frac{1}{16} \frac{\dd z_1}{(z_1+1)^3} - \left( \frac{1}{4} \zqm(2) - \frac{1}{32} \right) \frac{\dd z_1}{(z_1+1)^2}. \label{eq:basecases}
\end{align}

Now consider $\omega^M_{g,n}$ with $2g-2+n \geq 2$. By \cref{prop:principalparts}, we can express the topological recursion applied to the spectral curve of \cref{eq:Mspectralcurve} in the following way. Here, we have separated out the $\omega_{0,2}$ terms, used the symmetry of the stable correlation differentials appearing in \cref{eq:omegasymmetries}, and written out $\omega_{0,2}(z_1, z_2) = \frac{\dd z_1 \, \dd z_2}{(z_1-z_2)^2}$ explicitly. The word ``stable'' over the inner summation on the second line means that we exclude terms that include $\omega_{0,1}$ or $\omega_{0,2}$.
\begin{align} \label{eq:TRsimple}
\omega^M_{g,n}(z_1, \bm{z}_S) &= \sum_{\alpha = \pm 1} \mathcal{P} \Bigg[ \frac{-1}{2 \, y^M(z_1) \, x'(z_1)} \bigg( \frac{\omega^M_{g-1,n+1}(z_1, z_1, \bm{z}_S)}{\dd z_1} \bigg) \Bigg]_{z_1=\alpha} \notag \\
&+ \sum_{\alpha = \pm 1} \mathcal{P} \Bigg[ \frac{-1}{2 \, y^M(z_1) \, x'(z_1)} \bigg( \sum_{\substack{g_1+g_2=g \\ I \sqcup J = S}}^{\mathrm{stable}} \frac{\omega^M_{g_1,|I|+1}(z_1, \bm{z}_I) \, \omega^M_{g_2,|J|+1}(z_1, \bm{z}_J)}{\dd z_1} \bigg) \Bigg]_{z_1=\alpha} \notag \\
&+ \sum_{\alpha = \pm 1} \mathcal{P} \Bigg[ \frac{-1}{2 \, y^M(z_1) \, x'(z_1)} \sum_{j=2}^n \omega^M_{g,n-1}(z_1, \bm{z}_{S \setminus \{j\}}) \, \bigg( \frac{1}{(z_1-z_j)^2} + \frac{1}{(1-z_1z_j)^2} \bigg) \, \dd z_j \Bigg]_{z_1=\alpha} 
\end{align}

By the inductive hypothesis, the parenthesised term on the first line of \cref{eq:TRsimple} is an element of $\Lambda^M_{6g-8+2n}(z_1, \ldots, z_n)$. (The calculation of the degree here is given by $6(g-1) - 6 + 2(n+1) + 2$, where the addition of 2 comes from the fact that two of the arguments are set to $z_1$.) Multiply this by the series 
\[
\frac{-1}{2 \, y^M(z_1) \, x'(z_1)} = - \sum_{m=-2}^\infty A^+_m \big( \zqm(2), \zqm(4), \ldots \big) \, (z_1-1)^m,
\]
and take the principal part at $z_1 = 1$. Since $A^+_m$ has $\zeta$-degree that is even and at most $m+2$, the total degree is raised by at most 2, so we have that
\[
\mathcal{P} \Bigg[ \frac{-1}{2 \, y^M(z_1) \, x'(z_1)} \bigg( \frac{\omega^M_{g-1,n+1}(z_1, z_1, \bm{z}_S)}{\dd z_1} \bigg) \Bigg]_{z_1=1} \in \Lambda^M_{6g-6+2n}(z_1, \ldots, z_n).
\]
The same argument applies to the principal part at $z_1 = -1$. In fact, by the inductive hypothesis, the same argument applies to the entire second line of \cref{eq:TRsimple} as well.

To complete the proof, it remains to deal with the principal parts appearing on the third line of \cref{eq:TRsimple}. The inductive hypothesis implies that for $j = 2, 3, \ldots, n$,
\[
\omega^M_{g,n-1}(z_1, \bm{z}_{S \setminus \{j\}}) \in \Lambda^M_{6g-8+2n}(z_1, \ldots, \widehat{z}_j, \ldots, z_n),
\]
where $\widehat{z}_j$ denotes the fact that the argument $z_j$ has been omitted. To calculate the principal parts appearing on the third line of \cref{eq:TRsimple} at $z_1 = 1$ and $z_1 = -1$, we multiply $\omega^M_{g,n-1}(z_1, \bm{z}_{S \setminus \{j\}})$ by the respective expressions
\begin{align*}
\bigg( \frac{1}{(z_1-z_j)^2} + \frac{1}{(1-z_1z_j)^2} \bigg) \, \dd z_j &= \sum_{m=2}^\infty (-1)^m (m+1) \frac{z_j^m + (-1)^m}{(z_j-1)^{m+2}} \, (z_1-1)^m \, \dd z_j \\
\bigg( \frac{1}{(z_1-z_j)^2} + \frac{1}{(1-z_1z_j)^2} \bigg) \, \dd z_j &= \sum_{m=2}^\infty (m+1) \frac{z_j^m + (-1)^m}{(z_j+1)^{m+2}} \, (z_1+1)^m \, \dd z_j.
\end{align*}
This then produces an expression that lies in the vector space $\Lambda^M_{6g-8+2n}(z_1, \ldots, z_n)$, so the previous argument once again applies to the third line of \cref{eq:TRsimple}. This completes the induction and hence, the proof of the proposition.
\end{proof}

We now present a lemma that will be required in the next section.

\begin{lemma} \label{lem:equality}
\Cref{prop:preproof1} asserts that the stable correlation differentials can be expressed as
\[
\omega^M_{g,n}(z_1, \ldots, z_n) = \sum_{\substack{k_1, \ldots, k_n \geq 0 \\ k_1 + \cdots + k_n \leq 6g-6+2n \\ a_1, \ldots, a_n = \pm 1}} C^{a_1, \ldots, a_n}_{k_1, \ldots, k_n} \big( \zqm(2), \zqm(4), \ldots \big) \, \frac{\dd z_1}{(z_1-a_1)^{k_1+2}} \cdots \frac{\dd z_n}{(z_n-a_n)^{k_n+2}},
\]
where the $\zeta$-degree of $C^{a_1, \ldots, a_n}_{k_1, \ldots, k_n} \big( \zqm(2), \zqm(4), \ldots \big)$ is at most $6g-6+2n-\sum k_i$. In fact, unless $a_1, a_2, \ldots, a_n$ are all equal, the $\zeta$-degree of $C^{a_1, \ldots, a_n}_{k_1, \ldots, k_n} \big( \zqm(2), \zqm(4), \ldots \big)$ is strictly less than $6g-6+2n-\sum k_i$.
\end{lemma}

\begin{proof}
The proof involves a slightly more detailed analysis of the argument used for \cref{prop:preproof1}, so we provide an outline only. We proceed by induction and observe that the statement is true for the base cases $\omega_{0,3}$ and $\omega_{1,1}$, by inspection of \cref{eq:basecases}.

Now consider $\omega_{g,n}$ for $2g-2+n \geq 2$ which, following \cref{eq:TRsimple}, we can express as
\begin{align} \label{eq:firstterm}
\omega_{g,n}(z_1, \bm{z}_S) = \mathcal{P} \Bigg[ \frac{-1}{2 \, y^M(z_1) \, x'(z_1)} \bigg( \frac{\omega_{g-1,n+1}(z_1, z_1, \bm{z}_S)}{\dd z_1} \bigg) \Bigg]_{z_1=1} + \cdots.
\end{align}
We consider only the principal part of the first term on the right side of \cref{eq:TRsimple} at $z_1 = 1$, as our argument applies equally to the remaining terms and the principal parts at $z_1 = -1$.

Our aim is to prove that the total degree of a term 
\[
C \big( \zqm(2), \zqm(4), \ldots \big) \, \frac{\dd z_1}{(z_1-a_1)^{k_1+2}} \cdots \frac{\dd z_n}{(z_n-a_n)^{k_n+2}}
\]
in $\omega_{g,n}$ is $6g-6+2n$ only if $(a_1, a_2, \ldots, a_n) = (1, 1, \ldots, 1)$ or $(-1, -1, \ldots, -1)$. By the inductive hypothesis, the total degree of a term appearing in $\frac{1}{\dd z_1} \omega_{g-1,n+1}(z_1, z_1, \bm{z}_S)$ is at most $6g-8+2n$ with equality only if $(a_1, a_2, \ldots, a_n) = (1, 1, \ldots, 1)$ or $(-1, -1, \ldots, -1)$. As explained in the proof of \cref{prop:preproof1}, multiplying by $\frac{-1}{2 \, y^M(z_1) \, x'(z_1)}$ and taking the principal part produces terms in which the total degree is raised by at most 2. This shows that the terms arising from \cref{eq:firstterm} have total degree at most $6g-6+2n$ with equality only if $(a_1, a_2, \ldots, a_n) = (1, 1, \ldots, 1)$ or $(-1, -1, \ldots, -1)$. The remaining terms appearing in the expression for $\omega_{g,n}$ of \cref{eq:TRsimple} can be handled in a similar way.
\end{proof}

We now proceed with the proof of \cref{thm:Mqzeta}, from which \cref{thm:qzeta} concerning the appearance of $q$-zeta values in Okuyama's discrete volumes follows by passing to the large $M$ limit.

\begin{proof}[Proof of \cref{thm:Mqzeta}]
Part (a) of the theorem asserts that $N^M_{g,n}(b_1, \ldots, b_n)$ is quasi-polynomial in $b_1^2, \ldots, b_n^2$. This is an immediate consequence of \cref{thm:norsco}, which applies to any rational spectral curve with $x(z) = \alpha + \gamma(z + z^{-1})$.

From \cref{prop:preproof1}, we know that
\[
\omega^M_{g,n}(z_1, \ldots, z_n) = \sum_{\substack{k_1, \ldots, k_n \geq 0 \\ k_1 + \cdots + k_n \leq 6g-6+2n \\ a_1, \ldots, a_n = \pm 1}} C^{a_1, \ldots, a_n}_{k_1, \ldots, k_n} \big( \zqm(2), \zqm(4), \ldots \big) \, \frac{\dd z_1}{(z_1-a_1)^{k_1+2}} \cdots \frac{\dd z_n}{(z_n-a_n)^{k_n+2}},
\]
where $C^{a_1, \ldots, a_n}_{k_1, \ldots, k_n} \big( \zqm(2), \zqm(4), \ldots \big)$ is a polynomial of $\zeta$-degree at most $6g-6+2n-\sum k_i$. Now we simply expand the right side using
\[
\frac{1}{(z-1)^{k+2}} = (-1)^k \sum_{b=1}^\infty \binom{b+k}{k+1} z^{b-1} \qquad \text{and} \qquad
\frac{1}{(z+1)^{k+2}} = \sum_{b=1}^\infty (-1)^{b-1} \binom{b+k}{k+1} z^{b-1}.
\]
This leads to
\begin{align*}
\omega^M_{g,n}(z_1, \ldots, z_n) = \sum_{b_1, \ldots, b_n=1}^\infty \sum_{\substack{k_1, \ldots, k_n \geq 0 \\ k_1 + \cdots + k_n \leq 6g-6+2n \\ a_1, \ldots, a_n = \pm 1}} &C^{a_1, \ldots, a_n}_{k_1, \ldots, k_n} \big( \zqm(2), \zqm(4), \ldots \big) \\
&\prod_{i=1}^n (-1)^{f(a_i, b_i, k_i)} \, \frac{(b_i+1)(b_i+2) \cdots (b_i+k_i)}{(k_i+1)!} \, b_i z_i^{b_i-1} \, \dd z_i,
\end{align*}
where $f(1, b, k) = k$ and $f(-1, b, k) = b-1$. Comparing with \cref{eq:omegaMexpansion} then yields
\begin{align} \label{eq:quasipolynomial}
N^M_{g,n}(b_1, \ldots, b_n) = \sum_{\substack{k_1, \ldots, k_n \geq 0 \\ k_1 + \cdots + k_n \leq 6g-6+2n \\ a_1, \ldots, a_n = \pm 1}} &C^{a_1, \ldots, a_n}_{k_1, \ldots, k_n} \big( \zqm(2), \zqm(4), \ldots \big) \notag \\
&\prod_{i=1}^n (-1)^{f(a_i, b_i, k_i)} \, \frac{(b_i+1)(b_i+2) \cdots (b_i+k_i)}{(k_i+1)!}.
\end{align}
For fixed $a_1, \ldots, a_n$ and $k_1, \ldots, k_n$, it is clear that the product appearing here is a quasi-polynomial in $b_1, \ldots, b_n$ of degree $\sum k_i$. We know that the coefficient $C^{a_1, \ldots, a_n}_{k_1, \ldots, k_n} \big( \zqm(2), \zqm(4), \ldots \big)$ has $\zeta$-degree at most $6g-6+2n-\sum k_i$. So $N^M_{g,n}(b_1, \ldots, b_n)$ is a quasi-polynomial in which each underlying polynomial is an element of the graded ring
\[
\mathbb{Q}[\zqm(2), \zqm(4), \zqm(6), \ldots] [b_1, \ldots, b_n]
\]
of degree at most $6g-6+2n$. By \cref{thm:norsco}, we know that the only terms that survive the sum in \cref{eq:quasipolynomial} are quasi-polynomial in $b_1^2, \ldots, b_n^2$, which completes the proof of part (b) of the theorem.

Finally, part (c) of the theorem follows from the fact that the polynomial $C^{a_1, \ldots, a_n}_{k_1, \ldots, k_n}$ is defined independently of $M$, a consequence of the proof of \cref{prop:preproof1}.
\end{proof}

\subsection{Weil--Petersson volumes in the limit}

In the previous section, we worked with $N^M_{g,n}(b_1, \ldots, b_n)$ obtained from the spectral curve of \cref{eq:Mspectralcurve}, in which the product in the definition of $y(z)$ is truncated. The results of the previous section now allow us to pass to the large $M$ limit and work directly with $N^q_{g,n}(b_1, \ldots, b_n)$. Our goal in this section is to prove that one obtains the Weil--Petersson volume $V_{g,n}(L_1, \ldots, L_n)$ in a particular $\lambda \to 0$ limit involving $N^q_{g,n}(b_1, \ldots, b_n)$. This is essentially a consequence of the associated spectral curves of \cref{eq:okuyama,eq:WPvolumes} being related. We begin with the following observation.

\begin{lemma} \label{lem:kernellimit}
Using $z = \lambda t + 1$, we have
\[
\lim_{\lambda \to 0} \frac{\lambda^2}{2 \, y(z) \, x'(z)} = \frac{\pi}{2t \sin(2\pi t)}.
\]
\end{lemma}

\begin{proof}
Consider the following sequence of equalities.
\begin{align*}
\lim_{\lambda \to 0} \frac{\lambda^2}{2 \, y(z) \, x'(z)} &= \lim_{\lambda \to 0} \frac{\lambda^2 (\lambda t + 1)^3}{((\lambda t + 1)^2-1)^2} \, \exp \bigg[ \sum_{m=2}^\infty \sum_{k=1}^{\lfloor m/2 \rfloor} a_{mk} \, \zq(2k) \, \lambda^m t^m \bigg] \\
&= \frac{1}{4t^2} \lim_{\lambda \to 0} \exp \bigg[ \sum_{m=2}^\infty \sum_{k=1}^{\lfloor m/2 \rfloor} a_{mk} \, \zq(2k) \, \lambda^m t^m \bigg] \\
&= \frac{1}{4t^2} \exp \bigg[ \sum_{k=1}^\infty \frac{4^k}{k} \, \zeta(2k) \, t^{2k} \bigg] \\
&= \frac{\pi}{2 t \sin(2\pi t)}
\end{align*}
The first equality is a direct application of \cref{lem:expansions}, while the second explicitly calculates the limit of the prefactor. The third equality uses the fact that 
\[
\lim_{\lambda \to 0} \lambda^{2k} \zeta_q(2k) = \zeta(2k),
\]
which implies that most terms vanish in the limit, together with the evaluation $a_{2k,k} = \frac{4^k}{k}$. The fourth equality uses the series for $\log \big( \frac{\sin(x)}{x} \big)$ in terms of even zeta values.
\end{proof}

The previous lemma relates the topological recursion kernels associated to the spectral curves of \cref{eq:okuyama,eq:WPvolumes}. Naturally, one would then expect their stable correlation differentials to be related, which is the content of the following proposition.

\begin{proposition} \label{prop:omegalimit}
Using $z_i = \lambda t_i + 1$, the stable correlation differentials produced by Okuyama's spectral curve of \cref{eq:okuyama} are related to the stable correlation differentials produced by the spectral curve of \cref{eq:WPvolumes} via
\[
\lim_{\lambda \to 0} \lambda^{6g-6+3n} \, \omega_{g,n}(z_1, \ldots, z_n) = 2^{2-2g-n} \, \omega^{\mathrm{WP}}_{g,n}(t_1, \ldots, t_n).
\]
\end{proposition}

\begin{proof}
We proceed by induction on $2g-2+n$, with the base cases given explicitly by the following calculations, which use \cref{eq:basecases}.
\begin{align*}
\lim_{\lambda \to 0} \lambda^3 \, \omega_{0,3}(z_1, z_2, z_3) ={}& \lim_{\lambda \to 0} \left[ \frac{\lambda^3}{2} \frac{\lambda \, \dd t_1}{(\lambda t_1)^2} \frac{\lambda \, \dd t_2}{(\lambda t_2)^2} \frac{\lambda \, \dd t_3}{(\lambda t_3)^2} - \frac{\lambda^3}{2} \frac{\lambda \, \dd t_1}{(\lambda t_1+2)^2} \frac{\lambda \, \dd t_2}{(\lambda t_2+2)^2} \frac{\lambda \, \dd t_3}{(\lambda t_3+2)^2} \right] \\
={}& \frac{1}{2} \frac{\dd t_1}{t_1^2} \frac{\dd t_2}{t_2^2} \frac{\dd t_3}{t_3^2} \\
={}& \frac{1}{2} \, \omega^{\mathrm{WP}}_{0,3}(t_1, t_2, t_3) \\
~ \\
\lim_{\lambda \to 0} \lambda^3 \, \omega_{1,1}(z_1) ={}& \lim_{\lambda \to 0} \left[ \frac{\lambda^3}{16} \frac{\lambda \, \dd t_1}{(\lambda t_1)^4} + \frac{\lambda^3}{16} \frac{\lambda \, \dd t_1}{(\lambda t_1)^3} + \left( \frac{\lambda^3}{4} \zeta_q(2) - \frac{\lambda^3}{32} \right) \frac{\lambda \, \dd t_1}{(\lambda t_1)^2} \right] \\
&+ \lim_{\lambda \to 0} \left[ \frac{-\lambda^3}{16} \frac{\lambda \, \dd t_1}{(\lambda t_1+2)^4} + \frac{\lambda^3}{16} \frac{\lambda \, \dd t_1}{(\lambda t_1+2)^3} - \left( \frac{\lambda^3}{4} \zeta_q(2) - \frac{\lambda^3}{32} \right) \frac{\lambda \, \dd t_1}{(\lambda t_1+2)^2} \right] \\
={}& \frac{1}{16} \frac{\dd t_1}{t_1^4} + \frac{1}{4} \zeta(2) \frac{\dd t_1}{t_1^2} \\
={}& \frac{1}{2} \, \omega^{\mathrm{WP}}_{1,1}(t_1)
\end{align*}

Now consider $\omega_{g,n}$ for $2g-2+n \geq 2$ and invoke \cref{eq:TRsimple}.
\begin{align*}
& \lim_{\lambda \to 0} \lambda^{6g-6+3n} \, \omega_{g,n}(z_1, \zz_S) = \sum_{\alpha = \pm 1} \mathcal{P} \Bigg[ \bigg( \lim_{\lambda \to 0} \frac{-\lambda^2}{2 \, y(z_1) \, x'(z_1)} \bigg) \bigg( \lim_{\lambda \to 0} \lambda^{6g-8+3n} \frac{\omega_{g-1,n+1}(z_1, z_1, \bm{z}_S)}{\dd z_1} \bigg) \Bigg]_{z_1=\alpha} \\
&\qquad + \sum_{\alpha = \pm 1} \mathcal{P} \Bigg[ \bigg( \lim_{\lambda \to 0} \frac{-\lambda^2}{2 \, y(z_1) \, x'(z_1)} \bigg) \bigg( \lim_{\lambda \to 0} \lambda^{6g-8+3n} \sum_{\substack{g_1+g_2=g \\ I \sqcup J = S}}^{\mathrm{stable}} \frac{\omega_{g_1,|I|+1}(z_1, \bm{z}_I) \, \omega_{g_2,|J|+1}(z_1, \bm{z}_J)}{\dd z_1} \bigg) \Bigg]_{z_1=\alpha} \\
&\qquad + \sum_{\alpha = \pm 1} \mathcal{P} \Bigg[ \bigg( \lim_{\lambda \to 0} \frac{-\lambda^2}{2 \, y(z_1) \, x'(z_1)} \bigg) \sum_{j=2}^n \bigg( \lim_{\lambda \to 0} \lambda^{6g-9+3n} \omega_{g,n-1}(z_1, \bm{z}_{S \setminus \{j\}}) \bigg) \\
&\qquad\qquad\qquad\qquad\qquad\qquad\qquad\qquad\qquad \bigg( \lim_{\lambda \to 0} \lambda \bigg[ \frac{1}{(z_1-z_j)^2} + \frac{1}{(1-z_1z_j)^2} \bigg] \, \dd z_j \bigg) \Bigg]_{z_1=\alpha} 
\end{align*}

Observe that the principal parts at $z_1 = -1$ vanish in the limit. Use \cref{lem:kernellimit}, the inductive hypothesis and the fact that $\frac{1}{(z_1-z_k)^2} + \frac{1}{(1-z_1z_k)^2} = \lambda^{-2} \left( \frac{1}{(t_1-t_2)^2} + \frac{1}{(t_1+t_2)^2} \right) + O(\lambda^{-1})$ to obtain the following.
\begin{align*}
\lim_{\lambda \to 0} \lambda^{6g-6+3n} \, &\omega_{g,n}(z_1, \zz_S) = \mathcal{P} \Bigg[ \frac{-\pi}{2 t \sin(2\pi t)} \bigg( 2^{3-2g-n} \, \frac{\omega^{\mathrm{WP}}_{g-1,n+1}(t_1, t_1, \bm{t}_S)}{\dd t_1} \bigg) \Bigg]_{t_1=0} \\
&+ \mathcal{P} \Bigg[ \frac{-\pi}{2 t \sin(2\pi t)} \bigg( 2^{3-2g-n} \sum_{\substack{g_1+g_2=g \\ I \sqcup J = S}}^{\mathrm{stable}} \frac{\omega^{\mathrm{WP}}_{g_1,|I|+1}(t_1, \bm{t}_I) \, \omega^{\mathrm{WP}}_{g_2,|J|+1}(t_1, \bm{t}_J)}{\dd t_1} \bigg) \Bigg]_{t_1=0} \\
&+ \mathcal{P} \Bigg[ \frac{-\pi}{2 t \sin(2\pi t)} \sum_{j=2}^n 2^{3-2g-n} \, \omega^{\mathrm{WP}}_{g,n-1}(t_1, \bm{t}_{S \setminus \{j\}}) \bigg( \frac{1}{(t_1-t_j)^2} + \frac{1}{(t_1+t_j)^2} \bigg) \, \dd t_j \Bigg]_{t_1=0}
\end{align*}

Applying \cref{prop:principalparts} to the spectral curve of \cref{eq:WPvolumes}, in which $x(t) = \frac{t^2}{2}$ and $y(t) = \frac{1}{2\pi} \sin(2\pi t)$, equates the right side of the equation above with $2^{2-2g-n} \omega^{\mathrm{WP}}_{g,n}(t_1, \ldots, t_n)$. This completes the induction and the proof of the proposition.
\end{proof}

Now use \cref{prop:omegalimit} to relate Okuyama's discrete volumes with the Weil--Petersson volumes. This allows us to deduce \cref{thm:WPvolumes}, which states that for $(g,n) \neq (0,1)$ or $(0,2)$,
\[
\lim_{\lambda \to 0} \lambda^{6g-6+2n} \, {N}^q_{g,n} \Big( \frac{L_1}{\lambda}, \ldots, \frac{L_n}{\lambda} \Big) = 2^{3-2g-n} \, V_{g,n}(L_1, \ldots, L_n).
\]

\begin{proof}[Proof of \cref{thm:WPvolumes}]
Recall that \cref{prop:preproof1} allows us to write
\begin{equation} \label{eq:omegaexpansion}
\omega_{g,n}(z_1, \ldots, z_n) = \sum_{\substack{k_1, \ldots, k_n \geq 0 \\ k_1 + \cdots + k_n \leq 6g-6+2n \\ a_1, \ldots, a_n = \pm 1}} C^{a_1, \ldots, a_n}_{k_1, \ldots, k_n} \big( \zq(2), \zq(4), \ldots \big) \, \frac{\dd z_1}{(z_1-a_1)^{k_1+2}} \cdots \frac{\dd z_n}{(z_n-a_n)^{k_n+2}},
\end{equation}
where $C^{a_1, \ldots, a_n}_{k_1, \ldots, k_n} \big( \zq(2), \zq(4), \ldots \big)$ is a polynomial in $\zq(2), \zq(4), \ldots$ of $\zeta$-degree at most $6g-6+2n-\sum k_i$.

Now multiply both sides of \cref{eq:omegaexpansion} by $\lambda^{6g-6+3n}$ before taking the limit $\lambda \to 0$, while using $z_i = \lambda t_i + 1$.
\begin{align*}
& \lim_{\lambda \to 0} \lambda^{6g-6+3n} \, \omega_{g,n}(z_1, \ldots, z_n) \\
&= \lim_{\lambda \to 0} \lambda^{6g-6+3n} \, \sum_{\substack{k_1, \ldots, k_n \geq 0 \\ k_1 + \cdots + k_n \leq 6g-6+2n \\ a_1, \ldots, a_n = \pm 1}} C^{a_1, \ldots, a_n}_{k_1, \ldots, k_n} \big( \zq(2), \zq(4), \ldots \big) \, \frac{\lambda \, \dd t_1}{(\lambda t_1+1-a_1)^{k_1+1}} \cdots \frac{\lambda \, \dd t_n}{(\lambda t_n+1-a_n)^{k_n+2}} \\
&= \sum_{\substack{k_1, \ldots, k_n \geq 0 \\ k_1 + \cdots + k_n \leq 6g-6+2n}} \lim_{\lambda \to 0} \lambda^{6g-6+2n-\sum k_i} C^{1, \ldots, 1}_{k_1, \ldots, k_n} \big( \zq(2), \zq(4), \ldots \big) \, \frac{\dd t_1}{t_1^{k_1+2}} \cdots \frac{\dd t_n}{t_n^{k_n+2}}
\end{align*}

Combining this equation with \cref{prop:omegalimit} yields
\[
\frac{\omega_{g,n}^{\mathrm{WP}}(t_1, \ldots, t_n)}{2^{2g-2+n}} = \sum_{\substack{k_1, \ldots, k_n \geq 0 \\ k_1 + \cdots + k_n \leq 6g-6+2n}} \lim_{\lambda \to 0} \lambda^{6g-6+2n-\sum k_i} C^{1, \ldots, 1}_{k_1, \ldots, k_n} \big( \zq(2), \zq(4), \ldots \big) \, \frac{\dd t_1}{t_1^{k_1+2}} \cdots \frac{\dd t_n}{t_n^{k_n+2}}.
\]
\Cref{eq:WPvolumes} equates $\omega^{\mathrm{WP}}_{g,n}$ with a Laplace transform of the Weil--Petersson volumes, so we can apply the inverse Laplace transform to the previous equation to obtain\footnote{Since Mirzakhani proved that $V_{g,n}(L_1, \ldots, L_n)$ is a polynomial in $L_1^2, \ldots, L_n^2$, it follows that the limit in \cref{eq:limit1} vanishes unless $k_1, \ldots, k_n$ are all even~\cite{mir07a}.}
\begin{equation} \label{eq:limit1}
\frac{V_{g,n}(L_1, \ldots, L_n)}{2^{2g-2+n}} = \sum_{\substack{k_1, \ldots, k_n \geq 0 \\ k_1 + \cdots + k_n \leq 6g-6+2n}} \lim_{\lambda \to 0} \lambda^{6g-6+2n-\sum k_i} C^{1, \ldots, 1}_{k_1, \ldots, k_n} \big( \zq(2), \zq(4), \ldots \big) \prod_{i=1}^n \frac{L_i^{k_i}}{(k_i+1)!}.
\end{equation}

On the other hand, recall \cref{eq:quasipolynomial} from the proof of \cref{thm:Mqzeta}, which allows us to write
\begin{align*}
N^q_{g,n}(b_1, \ldots, b_n) = \sum_{\substack{k_1, \ldots, k_n \geq 0 \\ k_1 + \cdots + k_n \leq 6g-6+2n \\ a_1, \ldots, a_n = \pm 1}} &C^{a_1, \ldots, a_n}_{k_1, \ldots, k_n} \big( \zq(2), \zq(4), \ldots \big) \notag \\
&\prod_{i=1}^n (-1)^{f(a_i, b_i, k_i)} \, \frac{(b_i+1)(b_i+2) \cdots (b_i+k_i)}{(k_i+1)!},
\end{align*}
where $f(1, b, k) = k$ and $f(-1, b, k) = b-1$. So we can now take the desired limit as follows.
\begin{align*}
&\lim_{\lambda \to 0} \lambda^{6g-6+2n} N^q_{g,n} \Big( \frac{L_1}{\lambda}, \ldots, \frac{L_n}{\lambda} \Big) \\
&= \sum_{\substack{k_1, \ldots, k_n \geq 0 \\ k_1 + \cdots + k_n \leq 6g-6+2n \\ a_1, \ldots, a_n = \pm 1}} \lim_{\lambda \to 0} \lambda^{6g-6+2n - \sum k_i} C^{a_1, \ldots, a_n}_{k_1, \ldots, k_n} \big( \zq(2), \zq(4), \ldots \big) \prod_{i=1}^n (-1)^{f(a_i,b_i,k_i)} \frac{L_i^{k_i}}{(k_i+1)!}
\end{align*}

By \cref{lem:equality}, the limit in the equation above is equal to zero unless we have $(a_1, a_2, \ldots, a_n) = (1, 1, \ldots, 1)$ or $(-1, -1, \ldots, -1)$. The fact that $x(z)$ and $y(z)$ are both odd functions of $z$ leads to the symmetry $\omega_{g,n}(z_1, \ldots, z_n) = \omega_{g,n}(-z_1, \ldots, -z_n)$, which in turn implies that $C^{a_1, \ldots, a_n}_{k_1, \ldots, k_n} = (-1)^{\sum (k_i-1)} C^{-a_1, \ldots, -a_n}_{k_1, \ldots, k_n}$. These observations together allow us to express the desired limit as follows.
\begin{align}
&\lim_{\lambda \to 0} \lambda^{6g-6+2n} N^q_{g,n} \Big( \frac{L_1}{\lambda}, \ldots, \frac{L_n}{\lambda} \Big) \notag \\
&= 2 \sum_{\substack{k_1, \ldots, k_n \geq 0 \\ k_1 + \cdots + k_n \leq 6g-6+2n}} \lim_{\lambda \to 0} \lambda^{6g-6+2n - \sum k_i} C^{1, \ldots, 1}_{k_1, \ldots, k_n} \big( \zq(2), \zq(4), \ldots \big) \prod_{i=1}^n \frac{L_i^{k_i}}{(k_i+1)!} \label{eq:limit2}
\end{align}

Comparing \cref{eq:limit1,eq:limit2} produces the desired equality.
\end{proof}

\section{Top degree terms} \label{sec:top}

In this section, we prove that the top degree part of Okuyama's quasi-polynomials coincide precisely with the $q$-deformations of the Weil--Petersson volumes previously constructed by the authors~\cite{do-nor25}.

Okuyama's quasi-polynomial $N^q_{g,n}(b_1, \ldots, b_n)$ has degree $6g-6+2n$ and we denote its top degree part by
\[
N^{q,\mathrm{top}}_{g,n}(b_1, \ldots, b_n) := N^{q}_{g,n}(b_1, \ldots, b_n) - [\,\text{terms of degree less than $6g-6+2n$}\,].
\]
A priori, $N^{q,\mathrm{top}}_{g,n}(b_1, \ldots, b_n)$ is a quasi-polynomial, but we will show that it is in fact a polynomial, no longer dependent on the parities of $b_1, \ldots, b_n$. Furthermore, the set of polynomials $N^{q,\mathrm{top}}_{g,n}(b_1, \ldots, b_n)$ satisfy a recursion between themselves. These are consequences of \cref{lem:topdegTR} below. We will need the Laplace transform of a polynomial $P(x_1, \ldots, x_n)$, defined by 
\[
\cl\{P\}(z_1, \ldots, z_n) = \int_0^\infty \!\!\! \cdots \! \int_0^\infty P(x_1, \ldots, x_n) \prod_{i=1}^n \exp(-z_ix_i) \, \dd x_1 \cdots \dd x_n,
\]
for $\mathrm{Re}(z_i) > 0$. This extends to a meromorphic function on $\bc^n$ that is a polynomial in $z_i^{-1}$. We will show that the Laplace transforms of the polynomials $N^{q,\mathrm{top}}_{g,n}(b_1, \ldots, b_n)$ arise as correlation differentials of a particular spectral curve.

The spectral curve is constructed from Okuyama's spectral curve of \cref{eq:okuyama}, using only the top coefficients $a_{2k,k} = \frac{4^k}{k}$ of the coefficients $a_{mk}$ defined by the expansion of $\log(2 y^M(z) x'(z))$ in \cref{lem:expansions}. It is given by the data
\begin{equation} \label{eq:topcurve} 
\mathcal{S}^{\mathrm{top}} = \left( \mathbb{CP}^1, ~x(z) = \frac{1}{2} z^2, ~y^{\mathrm{top}}(z) = z \exp \bigg( -\sum_{m=1}^\infty \frac{\zeta_q(2m)}{m} (4z^2)^m \bigg), ~B(z, z') = \frac{\dd z \, \dd z'}{(z-z')^2} \right).
\end{equation}

Eynard and Orantin showed in \cite{eyn-ora09} that for a general spectral curve, the asymptotic behaviour -- equivalently, the largest order principal part -- of the correlation differential $\omega_{g,n}$ near a zero of $\dd x$ is given by the corresponding correlation differential $\omega^{\mathrm{Airy}}_{g,n}$ for the local model of the Airy spectral curve 
\begin{equation} \label{eq:airycurve} 
\mathcal{S}^{\mathrm{Airy}} = \left( \mathbb{CP}^1, ~x(z) = \frac{1}{2} z^2, ~y^{\mathrm{top}}(z) = z, ~B(z, z') = \frac{\dd z \, \dd z'}{(z-z')^2} \right).
\end{equation}
In particular, at the zero $\alpha$ of $\dd x$, we have
\[
\omega_{g,n} \sim c_\alpha^{2g-2+n} \, \omega^{\mathrm{Airy}}_{g,n},
\]
for $c_\alpha = \displaystyle\mathop{\mathrm{Res}}_\alpha \frac{\dd y \cdot \dd y}{\dd x}$. For Okuyama's spectral curve, one can calculate $c_1 = \frac{1}{2}$ and $c_{-1} = -\frac{1}{2}$.

The Airy spectral curve has correlation differentials that store intersection numbers of tautological classes on the moduli space of stable curves via the formula
\[
\omega_{g,n}^{\mathrm{Airy}} = \frac{\partial}{\partial z_1} \cdots \frac{\partial}{\partial z_n} \cl \bigg\{ \int_{\overline{\modm}_{g,n}} \!\!\! \exp \bigg( \frac{1}{2} \sum_{i=1}^n \psi_i L_i^2 \bigg) \bigg\} \, \dd z_1 \cdots \dd z_n.
\]
In particular, for a rational spectral curve with $x(z) = z+\frac{1}{z}$, the previous considerations show that the top degree part of the quasipolynomials are polynomials with known coefficients. The following lemma generalises this to the asymptotic $q\to1$ behaviour near $z=\pm1$ of the correlators of Okuyama's spectral curve.

\begin{lemma} \label{lem:topdegTR}
Topological recursion applied to the spectral curve $\mathcal{S}^{\mathrm{top}}$ of \cref{eq:topcurve} produces the correlation differentials
\[
\omega^{\mathrm{top}}_{g,n}(z_1, \ldots z_n) = 2^{2g-3+n} \frac{\partial}{\partial z_1} \cdots \frac{\partial}{\partial z_n} \cl \big\{ N^{q,\mathrm{top}}_{g,n}(b_1, \ldots, b_n) \big\} \, \dd z_1 \cdots \dd z_n.
\]
\end{lemma}

\begin{proof}
The highest order terms $\omega^{M,\mathrm{top}}_{g,n}(z_1, \ldots, z_n)$$\frac{\partial}{\partial z_1} \cdots \frac{\partial}{\partial z_n}\cl\{N^{q,\mathrm{top}}_{g,n}(b_1, \ldots, b_n)\} \, \dd z_1 \cdots \dd z_n$. Essentially we have $\omega_{g,n}^M\sim2\cdot(1/2)^{2g-2+n}\omega_{g,n}^{\mathrm{top}}$ where the scale $c_1=1/2$ introduces a factor of $(1/2)^{2g-2+n}$ and the factor of 2 comes from the initial cases $\omega_{0,3}^M=\omega_{0,3}^{\mathrm{top}}$ and $\omega_{1,1}^M=\omega_{1,1}^{\mathrm{top}}$.

In more detail, by homogeneity of the recursion, the top degree coefficients from the expansion from Lemma 3.3, given by $\exp \bigg( \sum_{j=1}^\infty\frac{\zeta_q(2j)}{j} (4z^2)^j \bigg)$, produces a recursion solely between the highest order terms $\omega^{M,\mathrm{top}}_{g,n}(z_1, \ldots, z_n)$ of $\omega^M_{g,n}(z_1, \ldots, z_n)$, which consist of highest order terms of $\omega^M_{g,n}(1+\frac{t_1}{1-q}, \ldots, 1+\frac{t_n}{1-q})$ as $q\to1$. The recursion of the principal parts around $z=1$ becomes the following. Put $s_j=z_j-1$ and $\bm{s}_K=(s_2, \ldots, s_n)$.
\begin{align} 
\omega^{M,\mathrm{top}}_{g,n}(s_1, \bm{s}_K)=&\mathcal{P}\left[\frac{1}{4s_1} \exp \bigg( \sum_{j=1}^\infty\frac{\zeta_q(2j)}{j} (4s_1^2)^j \bigg)\omega^{M,\mathrm{top}}_{g-1,n+1}(s_1, s_1, \bm{s}_K) \right]_{s_1=0}\\
&+\mathcal{P}\Big[\frac{1}{4s_1} \exp \bigg( \sum_{j=1}^\infty\frac{\zeta_q(2j)}{j} (4s_1^2)^j \bigg)\mathop{\sum_{g_1+g_2=g}}_{I \sqcup J = K}\hspace{-1mm}\omega^{M,\mathrm{top}}_{g_1,|I|+1}(s_1, \bm{s}_I) \, \omega^{M,\mathrm{top}}_{g_2,|J|+1}(s_1, \bm{s}_J) \Big]_{s_1=0}
\nonumber\\
&+\sum_{j=2}^n\mathcal{P}\left[\frac{1}{4s_1} \exp \bigg( \sum_{j=1}^\infty\frac{\zeta_q(2j)}{j} (4s_1^2)^j \bigg)\frac{\omega^{M,\mathrm{top}}_{g,n-1}(s_1,s_{K\setminus\{ j\}})}{(s_j-s_1)^2}\right]_{s_1=0}.
\nonumber
\end{align}
Topological recursion of the spectral curve $\mathcal{S}^{\mathrm{top}}$ in terms of principal parts has the same form, although with factor $\frac14s^{-1}\exp ( L_q(s) )$ replaced by $\frac{1}{2y^{\mathrm{top}}(s)x'(s)}=\frac12s^{-1}\exp (L_q(s))$ for $L_q(s)=\sum_{j=1}^\infty\frac{\zeta_q(2j)}{j} (4s^2)^j$.
This produces a factor of $2^{2g-2+n}$ between the correlators. The initial cases are calculated in the beginning of the proof of Proposition~\ref{prop:omegalimit} to give $\omega_{0,3}^M=\frac12\omega_{0,3}^{\mathrm{top}}$ and $\omega_{1,1}^M=\frac12\omega_{1,1}^{\mathrm{top}}$, and the initial conditions calculates $\omega_{0,3}^M$ and $\omega_{1,1}^M$
\end{proof}

\subsection{The $q$-deformation of Weil--Petersson volumes}

The $q$-deformation of Mirzakhani's recursion for Weil--Petersson volumes defined in \cite{do-nor25} produce polynomials $V^q_{g,n}(L_1, \ldots, L_n) \in\bq[[q]][L_1^2, \ldots, L_n^2]$ by replacing the function $H(x,y)$ in Mirzakhani's recursion \eqref{eq:mirzakhani} by
\begin{equation} \label{eq:Hq}
H_q(x,y) = \frac{1}{2} \sum_{m=1}^\infty(-1)^{m-1} q^{m^2/2} (q^{m/2} + q^{-m/2}) \Big(e^{\frac{1}{2}(x+y) (q^{m/2}-q^{-m/2})} + e^{\frac{1}{2}(x-y) (q^{m/2}-q^{-m/2})} \Big)
\end{equation}
and using the same initial conditions except for $V^q_{1,1}(L)=\frac{1}{48}L^2+\frac{1}{2}\zeta_q(2)$. To guarantee integrability, $H_q(x,y)$ is to be understood as a series in $y^2$.

The Laplace transform of the recursion \eqref{eq:mirzakhani} requires the Laplace transform of the linear transformations given by the double and single integrals in the recursion. Such Laplace transforms appeared in \cite[Lemma 6.10, 6.11]{nor20} and we repeat these calculations in the following two lemmas.

\begin{lemma} \label{lem:dint}
If $P(x,y)$ is a polynomial that is an odd function of both $x$ and $y$, then
\[
\cl \bigg\{\int_0^\infty \!\!\! \int_0^\infty \dd x \, \dd y \, H_q(x+y, L) \, P(x,y) \bigg\} = \mathcal{P} \bigg[ z^{-1} \exp \bigg( \sum_{j=1}^\infty\frac{\zeta_q(2j)}{j} (4z^2)^j \bigg) \cl\{ P \}(z,z) \bigg]_{z=0}.
\]
\end{lemma}

\begin{proof}
We prove the statement for $P(x,y) = \frac{x^{2i-1} \, y^{2j-1}}{(2i-1)! \, (2j-1)!}$, whose Laplace transform is $\cl\{ P \}(z_1, z_2) = \frac{1}{z_1^{2i} \, z_2^{2j}}$, and then extend by linearity to obtain the desired result. It is proven in \cite[Proposition~4]{do-nor25} that for each positive integer $k$,
\[
\frac{F_{2k-1}(y)}{(2k-1)!} := \int_0^\infty \!\! \frac{x^{2k-1}}{(2k-1)!} \, H_q(x,y) \, \dd x = \sum_{n=0}^{k} b_n \frac{y^{2k-2n}}{(2k-2n)!},
\]
where $b_0, b_1, b_2, \ldots \in \bq[\zeta_q(2), \zeta_q(4), \ldots]$ are defined by
\[
\sum_{n=0}^\infty b_n z^{2n} = \exp \bigg(\sum_{m=1}^\infty\frac{\zeta_q(2m)}{m} (4z^2)^m \bigg) = 1 + 4\zeta_q(2) z^2 + 8(\zeta_q(2)^2+\zeta_q(4)) z^4 + \cdots.
\]

Use the change of coordinates $u=x+y$ and $v=x$ to obtain
\[
\int_0^\infty \!\!\! \int_0^\infty \frac{x^{2i-1} \, y^{2j-1}}{(2i-1)! \, (2j-1)!} \, H_q(x+y,L) \, \dd x \, \dd y = \frac{F_{2i+2j-1}(L)}{(2i+2j-1)!} = \sum_{n=0}^{i+j} b_{i+j-n} \frac{L^{2n}}{(2n)!}.
\]
Hence, its Laplace transform is
\[
\cl \bigg\{ \int_0^\infty \!\!\! \int_0^\infty \frac{x^{2i-1} \, y^{2j-1}}{(2i-1)! \, (2j-1)!} H_q(x+y,L) \, \dd x \, \dd y \bigg\} = \sum_{n=0}^{i+j}\frac{b_{i+j-n}}{z^{2n+1}},
\]
which coincides with the principal part of 
\[
z^{-1} \exp \bigg( \sum_{m=1}^\infty \frac{\zeta_q(2m)}{m} (4z^2)^m \bigg) \cl \{ P \}(z,z) \sim \frac{1}{z^{2i+2j+1}} \sum_{n=0}^{\infty} b_n z^{2n},
\] 
where $\sim$ means the equality of Laurent series at $z=0$.
\end{proof}

\begin{lemma} \label{lem:sint}
If $P(x)$ is a polynomial that is an odd function of $x$, then
\[
\cl \bigg\{ \frac{1}{2} \int_0^\infty \dd x \, \big( H_q(L_1+L_2, x) + H_q(L_1-L_2, x) \big) \, P(x) \bigg\} = \mathcal{P} \bigg[ z_1^{-1} \exp \bigg( \sum_{j=1}^\infty \frac{\zeta_q(2j)}{j} (4z_1^2)^j \bigg) \frac{\cl\{P\}(z_1)}{z_2-z_1} \bigg]_{z_1=0}^-.
\]
Here, we use $\mathcal{P}[\omega(z)]^-_{z=\alpha}$ to denote the odd part of the principal part under the involution $z \mapsto -z$.
\end{lemma}

\begin{proof}
We prove the statement for $P(x) = \frac{x^{2k-1}}{(2k-1)!}$, whose Laplace transform is $\cl \{P\}(z) = \frac{1}{z^{2k}}$, and then extend by linearity to obtain the desired result. 

Then
\begin{align*}
\frac12\int_0^\infty \hspace{-2mm} \dd x (H_q(L_1+L_2,x)+H_q(L_1-L_2,&x))\frac{x^{2k-1}}{(2k-1)!}=\frac12\frac{F_{2k-1}(L_1+L_2)}{(2k-1)!}+\frac12\frac{F_{2k-1}(L_1-L_2)}{(2k-1)!}\\
&=\frac12\sum_{\epsilon=\pm1}\sum_{m=0}^{k}\frac{(L_1+\epsilon L_2)^{2m}}{(2m)!}b_{k-m}
=\hspace{-2mm}\sum_{i+j\leq k}\frac{L_1^{2i}L_2^{2j}}{(2i)!(2j)!}b_{k-i-j}.
\end{align*}
Hence its Laplace transform is:
\[
\cl\left\{\frac12\int_0^\infty \hspace{-2mm} \dd x (H_q(L_1+L_2,x)+H_q(L_1-L_2,x))\frac{x^{2k-1}}{(2k-1)!}\right\}=\sum_{i+j\leq k}\frac{1}{z_1^{2i+1}z_2^{2j+1}}b_{k-i-j}
\]
which coincides with the odd principal part in $z_1$ of 
\[
z_1^{-1}\exp \bigg( \sum_{m=1}^\infty\frac{\zeta_q(2m)}{m} (4z_1^2)^m \bigg)\frac{\cl\left\{\frac{x^{2k-1}}{(2k-1)!}\right\}(z_1)}{(z_2-z_1)}\sim \sum_{n=0}^{\infty}b_nz_1^{2n-1}\sum_{\ell=0}^\infty\frac{z_1^\ell}{z_2^{\ell+1}}\frac{1}{z_1^{2k}}
\]
where $\sim$ means the Laurent series at $z_1=0$ for fixed $z_2$, and $|z_1|<|z_2|$.
\end{proof}

The following proposition is analogous to the result by Eynard and Orantin \cite{eyn-ora07b} that the spectral curve $x=\frac12z^2$, $y=\frac{\sin(2\pi z)}{2\pi}$ stores the Weil--Petersson volumes.

\begin{proposition} \label{specqvol}
Topological recursion applied to the spectral curve $\mathcal{S}^{\mathrm{top}}$ of \cref{eq:topcurve} produces the correlation differentials
\[
\omega_{g,n} = \frac{\partial}{\partial z_1} \cdots \frac{\partial}{\partial z_n} \cl \big\{ V^q_{g,n}(L_1, \ldots, L_n) \big\} \, \dd z_1 \cdots \dd z_n.
\]
\end{proposition}

\begin{proof}
Take the Laplace transform of the $q$-analogue of Mirzakhani's recursion~\cite{do-nor25}, which replaces the kernel $H(x,y)$ in \eqref{eq:mirzakhani} with the $q$-deformed kernel $H_q(x,y)$ of \cref{eq:Hq}.
\begin{align} \label{eq:laprec}
& \cl \left\{ \frac{\partial}{\partial L_1} L_1 V^q_{g,n}(L_1, \bm{L}_K) \right\} \\
={}& \cl \bigg\{ \frac{1}{2} \int_0^\infty \!\! \int_0^\infty xy \, H_q(x+y,L_1)
\Big(V^q_{g-1,n+1}(x, y, \bm{L}_K) + \sum V^q_{g_1,|I|+1}(x,L_I)V^q_{g_2,|J|+1}(y,L_J)\Big) \dd x \, \dd y \nonumber\\
&\qquad + \frac{1}{2} \sum_{j=2}^n \int_0^\infty \hspace{-2mm}x(H_q(L_1+L_j,x)+H_q(L_1-L_j,x))V^q_{g,n-1}(x,L_{K\setminus\{j\}}) \dd x \bigg\} \notag \\
=& \frac{1}{2} \mathcal{P} \Big[z^{-1} \exp \bigg( \sum_{j=1}^\infty\frac{\zeta_q(2j)}{j} (4z^2)^j \bigg)\Big(\cl\{xyV^q_{g-1,n+1}\}(z_1,z_1,z_K)\nonumber\\
&\hspace{2cm}+\hspace{-2mm}\mathop{\sum_{g_1+g_2=g}}_{I \sqcup J = K}\hspace{-1mm}\cl\{xV^q_{g_1,|I|+1}\}(z_1,z_I)\cl\{yV^q_{g_2,|J|+1}\}(z_1,z_J)\Big)\Big]_{z_1=0} \nonumber\\
&\qquad+\sum_{j=2}^n\mathcal{P}\left[z^{-1} \exp \bigg( \sum_{j=1}^\infty\frac{\zeta_q(2j)}{j} (4z^2)^j \bigg)\frac{\cl\{xV^q_{g,n-1}\}(z_1,z_{K\setminus\{ j\}})}{z_j-z_1}\right]_{z_1=0}^- \nonumber.
\end{align}
which uses \cref{lem:dint,lem:sint}.

\begin{align}
\cl \bigg\{ \frac{\partial}{\partial L_1} L_1 V_{g,n}(L_1, \LL_S) \bigg\} &= \frac{1}{2} \cl \bigg\{ \int_0^\infty \!\!\! \int_0^\infty xy \, H_q(x+y, L_1) \, V_{g-1,n+1}(x, y, \LL_S) \, \dd x \, \dd y \bigg\} \notag \\
&+ \frac{1}{2} \sum_{\substack{g_1+g_2=g \\ I \sqcup J = S}} \cl \bigg\{ \int_0^\infty \!\!\! \int_0^\infty xy \, H_q(x+y,L_1) \, V_{g_1,|I|+1}(x, \LL_I) \, V_{g_2,|J|+1}(y, \LL_J) \, \dd x \ \dd y \bigg\} \notag \\
&+ \frac{1}{2} \sum_{k=2}^n \cl \bigg\{ \int_0^\infty x \big[ H_q(x, L_1+L_k) + H_q(x, L_1-L_k) \big] \, V_{g,n-1}(x, \LL_{S \setminus \{k\}}) \, \dd x \bigg\} \notag \\
&= \frac{1}{2} \mathcal{P} \bigg[z^{-1} \exp \bigg( \sum_{j=1}^\infty\frac{\zeta_q(2j)}{j} (4z^2)^j \bigg)\Big(\cl\{xyV^q_{g-1,n+1}\}(z_1,z_1,z_K) \notag\\
&\hspace{2cm}+\hspace{-2mm}\mathop{\sum_{g_1+g_2=g}}_{I \sqcup J = K}\hspace{-1mm}\cl\{xV^q_{g_1,|I|+1}\}(z_1, \bm{z}_I)\ cl \{ yV^q_{g_2,|J|+1}\}(z_1, \b,{z}_J) \Big) \bigg]_{z_1=0} \notag \\
&\qquad+\sum_{j=2}^n\mathcal{P}\left[z^{-1} \exp \bigg( \sum_{j=1}^\infty\frac{\zeta_q(2j)}{j} (4z^2)^j \bigg)\frac{\cl\{xV^q_{g,n-1}\}(z_1,z_{K\setminus\{ j\}})}{z_j-z_1}\right]_{z_1=0}^-.
\end{align}

Define 
\[
\bar{\omega}_{g,n}=(-1)^n\frac{\partial}{\partial z_1} \cdots \frac{\partial}{\partial z_n}\cl\{V^q_{g,n}(L_1, \ldots, L_n)\} \, \dd z_1 \cdots \dd z_n.
\] 
We will prove that $\bar\omega_{g,n}$ and the correlators $\omega_{g,n}$ satisfy the same recursion relations and initial values, and in particular conclude that $\bar\omega_{g,n}=\omega_{g,n}$.

Take $(-1)^{n-1}\frac{\partial}{\partial z_2} \cdots \frac{\partial}{\partial z_n}\Big[$\cref{eq:laprec}$\Big] \, \dd z_1 \cdots \dd z_n$, and use $\frac{\partial}{\partial z_1}\cl\{P(z_1)\}=-z^{-1}\cl\{\frac{\partial}{\partial L_1}L_1P(L_1)\}$,
 to get
\begin{align} 
\bar\omega_{g,n}(z_1, \ldots, z_n)=&\frac12\mathcal{P}\left[z_1^{-2} \exp \bigg( \sum_{j=1}^\infty\frac{\zeta_q(2j)}{j} (4z_1^2)^j \bigg)\bar\omega_{g-1,n+1}(z_1,z_1,z_K)\right]_{z_1=0}\\
&+\frac12\mathcal{P}\Big[z_1^{-2} \exp \bigg( \sum_{j=1}^\infty\frac{\zeta_q(2j)}{j} (4z_1^2)^j \bigg)\mathop{\sum_{g_1+g_2=g}}_{I \sqcup J = K}\hspace{-1mm}\bar\omega_{g_1,|I|+1}(z_1,z_I)\bar\omega_{g_2,|J|+1}(z_1,z_J)\Big]_{z_1=0}
\nonumber\\
&+\sum_{j=2}^n\mathcal{P}\left[z_1^{-2} \exp \bigg( \sum_{j=1}^\infty\frac{\zeta_q(2j)}{j} (4z_1^2)^j \bigg)\frac{\bar\omega_{g,n-1}(z_1,z_{K\setminus\{ j\}})}{(z_j-z_1)^2}\right]_{z_1=0}^-.
\nonumber
\end{align}
We have used $z_1^{-1}[F(z_1) \dd z_1]_{z_1=0}=[z_1^{-1}F(z_1) \dd z_1]_{z_1=0}$ which holds because the residue of $\bar\omega_{g,n}(z_1, \ldots, z_n)$ at $z_1=0$ vanishes. The factors $xy$, $x$ and $y$ on the right hand side of \cref{eq:laprec} supply derivatives such as $\cl\{xyV^q_{g-1,n+1}\}(z_1,z_1,z_K)=\frac{\partial^2}{\partial w\partial z}\cl\{V^q_{g-1,n+1}\}(w\hspace{-1mm}=\hspace{-1mm}z_1,z\hspace{-1mm}=\hspace{-1mm}z_1,z_K)$.

Topological recursion for the spectral curve $\mathcal{S}^{\mathrm{top}}$ is 
\begin{align*}
\omega_{g,n}(z_1,z_K)&=\Res_{z=0}K(z_1,z) \cf(\{\omega_{g',n'}(z,z_K)\}) \dd z \, \dd z \, \dd z_K \\
&=-\frac12\Res_{z=0}\left(\frac{\dd z_1}{z_1-z} - \frac{\dd z_1}{z_1+z}\right)\frac12z^{-2} \exp \bigg( \sum_{j=1}^\infty\frac{\zeta_q(2j)}{j} (4z^2)^j \bigg) \cf(\{\omega_{g',n'}(z,z_K)\}) \dd z \, \dd z_K\\
&=-\frac12\mathcal{P}\left[z^{-2} \exp \bigg( \sum_{j=1}^\infty\frac{\zeta_q(2j)}{j} (4z^2)^j \bigg) \cf(\{\omega_{g',n'}(z_1,z_K)\}) \dd z_1 \, \dd z_K \right]_{z_1=0}
\end{align*}
where $\cf(z_1,z_K)$ is a rational function given explicitly by
\begin{align*}
\cf(z_1,z_K) \dd z_1^2 \, \dd z_K = &\omega_{g-1,n+1}(z,-z,z_L)+ \hspace{-2mm}\mathop{\sum_{g_1+g_2=g}}_{I\sqcup J=L}^{\rm stable} \omega_{g_1,|I|+1}(z,z_I) \, \omega_{g_2,|J|+1}(-z,z_J) \\
+\sum_{j=2}^n&\big(\omega_{0,2}(z,z_j) \, \omega_{g,n-1}(-z,z_{K\setminus\{ j\}})+\omega_{0,2}(-z,z_j) \, \omega_{g,n-1}(z,z_{K\setminus\{ j\}})\big)\\
=&-\omega_{g-1,n+1}(z,z,p_L)- \hspace{-2mm}\mathop{\sum_{g_1+g_2=g}}_{I\sqcup J=L}^{\rm stable} \omega_{g_1,|I|+1}(z,z_I) \, \omega_{g_2,|J|+1}(z,z_J) \\
&-\sum_{j=2}^n\big(\omega_{0,2}(z,z_j)-\omega_{0,2}(-z,z_j)\big)\omega_{g,n-1}(z,z_{K\setminus\{ j\}})
\end{align*}
where we have used skew-symmetry of $\omega_{g,n}$ under $z_i\mapsto-z_i$, except for $\omega_{0,2}$. Hence,
\begin{align*} 
\omega_{g,n}(z_1,z_K)=&\frac12\mathcal{P}\left[z^{-2} \exp \bigg( \sum_{j=1}^\infty\frac{\zeta_q(2j)}{j} (4z^2)^j \bigg)\omega_{g-1,n+1}(z_1,z_1,z_K)\right]_{z_1=0}\\
&+\frac12\mathcal{P}\Big[z^{-2} \exp \bigg( \sum_{j=1}^\infty\frac{\zeta_q(2j)}{j} (4z^2)^j \bigg)\mathop{\sum_{g_1+g_2=g}}_{I \sqcup J = K}^{\rm stable}\hspace{-1mm}\omega_{g_1,|I|+1}(z_1,z_I)\omega_{g_2,|J|+1}(z_1,z_J)\Big]_{z_1=0}\\
&+\sum_{j=2}^n\mathcal{P}\left[z^{-2} \exp \bigg( \sum_{j=1}^\infty\frac{\zeta_q(2j)}{j} (4z^2)^j \bigg)\frac{\omega_{g,n-1}(z_1,z_{K\setminus\{ j\}})}{(z_j-z_1)^2}\right]_{z_1=0}^-.
\end{align*}
where we have used $[\omega_{0,2}(-z,z_j)\eta(z)]_{z=0}^-=-[\omega_{0,2}(z,z_j)\eta(z)]_{z=0}^-$ for $\eta(z)$ odd.

The rational differentials $\bar\omega_{g,n}$ and $\omega_{g,n}$ are uniquely determined by their respective recursions and the initial value
\[\bar\omega_{1,1}(z_1)=-\frac{\partial}{\partial z_1}\cl\{V^q_{1,1}(L_1)\}\dd z_1 = -\frac{\partial}{\partial z_1}\cl\Big\{\frac{1}{48}L^2+\frac{1}{2}\zeta_q(2)\Big\} \dd z_1 = \Big(\frac{1}{8z^4}+\frac{1}{2z_1^2}\zeta_q(2)\Big) \dd z_1 = \omega_{1,1}(z_1)\]
which both coincide, hence $\bar\omega_{g,n}=\omega_{g,n}$ as required.
\end{proof}
Theorem~\ref{thm:topdegterms} is an immediate consequence of Lemma~\ref{lem:topdegTR} and Proposition~\ref{specqvol}.

\subsection{The $q$-deformation of classical Weil--Petersson volumes}

In general, the correlators of a spectral curve can be expressed via intersection numbers of tautological classes over moduli spaces of stable curves. Eynard proved that for the rational spectral curve given by
\[
x(z) = \frac{1}{2} z^2 \qquad \text{and} \qquad y(z) = z + \sum_{k=1}^{\infty}y_k z^{2k+1},
\]
the correlation differentials are given by the following~\cite{eyn14}.
\[
\omega_{g,n} = \frac{\partial}{\partial z_1} \cdots \frac{\partial}{\partial z_n} \cl \bigg\{ \int_{\overline{\modm}_{g,n}} \!\!\! \exp \bigg(\sum_{m=1}^\infty s_m \kappa_m \bigg) \exp \bigg( \frac{1}{2} \sum_{i=1}^n \psi_i L_i^2 \bigg) \bigg\} \, \dd z_1 \cdots \dd z_n
\] 
Here, $s_1, s_2, s_3, \ldots$ are defined as functions of $y_1, y_2, y_3, \ldots$ by the formula
\[
\exp \left( -\sum_{m=1}^\infty s_m \lambda^m \right) = \sum_{k=0}^\infty (2k+1)!! \, y_k \lambda^k = \frac{1}{\sqrt{2\pi\lambda^3}} \int_{-\infty}^\infty \dd z \cdot zy(z) \, e^{-\frac{z^2}{2\lambda}}.
\]
For $y^{\mathrm{top}}(z) = z \exp \bigg( -\displaystyle\sum_{m=1}^\infty \dfrac{\zeta_q(2m)}{m} (4z^2)^m \bigg)$, we have
\[
\exp \left( -\sum_{m=1}^\infty s_m(q) \lambda^m \right) = \frac{1}{\sqrt{2\pi\lambda^3}} \int_{-\infty}^\infty \dd z \cdot z^2 \exp \bigg( -\sum_{m=1}^\infty \frac{\zeta_q(2m)}{m} (4z^2)^m \bigg) \, e^{-\frac{z^2}{2\lambda}},
\] 
which defines $\Omega_q(\kappa_1,\kappa_2, \ldots) \in H^*(\overline{\modm}_{g,n}; \bq[[q]])$ by
\begin{align*}
\Omega_q(\kappa_1, \kappa_2, \ldots) &= \sum_{m=1}^\infty s_m(q) \, \kappa_m \\
&= 12 \zeta_q(2) \kappa_1 + 24 (5\zeta_q(4) - 2\zeta_q(2)^2) \kappa_2 + 64 (35\zeta_q(6) - 30\zeta_q(4) \zeta_q(2) + 4 \zeta_q(2)^3) \kappa_3 + \cdots.
\end{align*}
Thus, we obtain the following expression for the $q$-deformed Weil--Petersson volume in terms of higher Mumford volumes.
\[
V^q_{g,n}(L_1, \ldots, L_n) = \int_{\overline{\modm}_{g,n}} \exp \bigg( \Omega_q(\kappa_1,\kappa_2, \ldots) + \frac{1}{2} \sum_{i=1}^n L_i^2 \psi_i \bigg)
\]

The formula above leads to the following definition of the $q$-deformed Weil--Petersson volumes for moduli spaces of stable curves without marked points.
\[
V_g(q) := \int_{\overline{\modm}_g} \exp \big( \Omega_q(\kappa_1,\kappa_2, \ldots) \big) \in \bq[[q]]
\]
For the simplest case $g=2$, we have
\[
V_2(q) = \frac{191}{90} \zeta_q(2)^3 + \frac{13}{3} \zeta_q(2) \zeta_q(4) + \frac{35}{18} \zeta_q(6),
\]
which satisfies
\[
\lim_{q \to 1} (1-q)^6 \, V_2(q) = \frac{191}{90} \zeta(2)^3 + \frac{13}{3} \zeta(2) \zeta(4) + \frac{35}{18} \zeta(6) = \frac{43\pi^6}{2160} = \mathrm{Vol}^{\mathrm{WP}}(\cm_2).
\]
For $g=3$, we have
\begin{align*}
V_3(q)=&\tfrac{10312177}{11340}\zeta_q(2)^6 + \tfrac{3829529}{756}\zeta_q(2)^4\zeta_q(4) + \tfrac{452120}{81}\zeta_q(2)^3\zeta_q(6) + \tfrac{1355537}{252}\zeta_q(2)^2\zeta_q(4)^2 + \tfrac{31453}{6}\zeta_q(2)^2\zeta_q(8)\\
+ \tfrac{151348}{27}&\zeta_q(2)\zeta_q(4)\zeta_q(6)+ \tfrac{51128}{15}\zeta_q(2)\zeta_q(10)+ \tfrac{155395}{252}\zeta_q(4)^3 + \tfrac{10085}{6}\zeta_q(4)\zeta_q(8) + \tfrac{54950}{81}\zeta_q(6)^2 + \tfrac{10010}{9}\zeta_q(12),
\end{align*}
which satisfies
\begin{align*}
\lim_{q \to 1} (1-q)^{12} \, V_3(q) &= \tfrac{10312177}{11340}\zeta(2)^6 + \tfrac{3829529}{756}\zeta(2)^4\zeta(4) + \tfrac{452120}{81}\zeta(2)^3\zeta(6) + \tfrac{1355537}{252}\zeta(2)^2\zeta(4)^2 + \cdots \\
&= \frac{176557\pi^{12}}{1209600} = \mathrm{Vol}^{\mathrm{WP}}(\cm_3).
\end{align*}

\appendix

\section{Table of $q$-Weil--Petersson volumes } \label{app:WPvolumes}

Following are several examples of $q$-Weil--Petersson volumes and the corresponding Weil--Petersson volumes. In each case, the $q \to 1$ limit can be explicitly observed. In some cases, we evaluate at $\LL=\mathbf{0}$. 
\begin{align*}
V^{\mathrm{WP}}_{0,3}(\LL) &= 1 \\
V^q_{0,3}(\LL) &= 1 \\
V^{\mathrm{WP}}_{0,4}(\LL) &= \frac{1}{2}(L_1^2+L_2^2+L_3^2+L_4^2) + 2\pi^2 \\
V^q_{0,4}(\LL) &= \frac{1}{2}(L_1^2+L_2^2+L_3^2+L_4^2) + 12\zeta_q(2) \\
V^{\mathrm{WP}}_{0,5}(\LL) &= \frac{1}{8}(L_1^4 + \cdots + L_5^4) + \frac{1}{2}(L_1^2L_2^2 + \cdots + L_4^2L_5^2) + 3\pi^2(L_1^2 + \cdots + L_5^2) + 10\pi^4 \\
V_{0,5}^q(\mathbf{0})&=312\zeta_q(2)^2+120\zeta_q(4)\\
V^{\mathrm{WP}}_{1,1}(\LL) &= \frac{1}{48}L_1^2 + \frac{\pi^2}{12} \\
V^q_{1,1}(\LL) &= \frac{1}{48}L_1^2 + \frac12\zeta_q(2) \\
V^{\mathrm{WP}}_{1,2}(\LL) &= \frac{1}{192}(L_1^4+L_2^4) + \frac{1}{96} L_1^2L_2^2 + \frac{\pi^2}{12}(L_1^2+L_2^2) + \frac{\pi^4}{4} \\
V_{1,2}^q(\LL)&=\frac{1}{192}(L_1^2+L_2^2)^2 + \frac{1}{2}(L_1^2+L_2^2)\zeta_q(2) + 7\zeta_q(2)^2 + 5\zeta_q(4)\\
V^{\mathrm{WP}}_{2,1}(\LL) &= \frac{1}{4423680} L_1^8 + \frac{29\pi^2}{138240} L_1^6 + \frac{139\pi^4}{23040} L_1^4 + \frac{169\pi^6}{2880} L_1^2 + \frac{29\pi^8}{192}\\
V_{2,1}^q(\LL)&=\frac{1}{442368}L_1^8 + \frac{29}{23040}L_1^6\zeta_q(2) + \frac{1}{1920}(359\zeta_q(2)^2 + 145\zeta_q(4))L_1^4\\
&\quad + \frac{1}{24}(191\zeta_q(2)^3 + 243\zeta_q(2)\zeta_q(4) + 70\zeta_q(6))L_1^2\\
&\quad +\frac{845}{12}\zeta_q(2)^4 + \frac{399}{2}\zeta_q(2)^2\zeta_q(4) + \frac{185}{4}\zeta_q(4)^2 + \frac{406}{3}\zeta_q(2)\zeta_q(6) + \frac{105}{2}\zeta_q(8)\\
V^{\mathrm{WP}}_{2,2}(\LL) &= \frac{1}{4423680} (L_1^{10}+L_2^{10}) + \frac{1}{294912}(L_1^8L_2^2 + L_2^8L_1^2) + \frac{29}{2211840}(L_1^6L_2^4+L_2^6L_1^4) + \frac{11\pi^2}{276480}(L_1^8+L_2^8) \\
V_{2,2}^q(\mathbf{0})&=\frac{247429}{15}\zeta_q(2)^5 - \frac{67844}{3}\zeta_q(2)^3\zeta_q(4) + 8525\zeta_q(2)\zeta_q(4)^2 + \frac{46228}{3}\zeta_q(2)^2\zeta_q(6) + \frac{13580}{3}\zeta_q(4)\zeta_q(6)\\
&\quad+ \frac{29\pi^2}{69120}(L_1^6L_2^2+L_2^6L_1^2) + \frac{7\pi^2}{7680}L_1^4L_2^4 + \frac{19\pi^4}{7680}(L_1^6+L_2^6) + \frac{181\pi^4}{11520}(L_1^4L_2^2+L_2^4L_1^2) \\
&\quad+ \frac{551\pi^6}{8640}(L_1^4+L_2^4) + \frac{7\pi^6}{36} L_1^2L_2^2 + \frac{1085\pi^8}{1728}(L_1^2+L_2^2) + \frac{787\pi^{10}}{480} \\
V_{0,5}^q(\mathbf{0})&=312\zeta_q(2)^2+120\zeta_q(4) \\
V_{0,6}^q(\mathbf{0})&=17824 \zeta_q(2)^3 - 1920 \zeta_q(2) \zeta_q(4) + 2240 \zeta_q(6)
\end{align*}

\bibliographystyle{plain}
\bibliography{DSSYK-to-mirzakhani}

\begin{thebibliography}{10}

\bibitem{ACEH20}
A.~Alexandrov, G.~Chapuy, B.~Eynard, and J.~Harnad.
\newblock Weighted {H}urwitz numbers and topological recursion.
\newblock {\em Comm. Math. Phys.}, 375(1):237--305, 2020.

\bibitem{ABCO24}
J\o rgen~Ellegaard Andersen, Ga\"{e}tan Borot, Leonid~O. Chekhov, and Nicolas
  Orantin.
\newblock The {ABCD} of topological recursion.
\newblock {\em Adv. Math.}, 439:Paper No. 109473, 105, 2024.

\bibitem{BBNR21}
Micha Berkooz, Nadav Brukner, Vladimir Narovlansky, and Amir Raz.
\newblock Multi-trace correlators in the {SYK} model and non-geometric
  wormholes.
\newblock {\em J. High Energy Phys.}, (9):Paper No. 196, 68, 2021.

\bibitem{BINT19}
Micha Berkooz, Mikhail Isachenkov, Vladimir Narovlansky, and Genis Torrents.
\newblock Towards a full solution of the large {N} double-scaled {SYK} model.
\newblock {\em J. High Energy Phys.}, (3):079, 70, 2019.

\bibitem{BCGLS21}
Ga\"{e}tan Borot, S\'{e}verin Charbonnier, Elba Garcia-Failde, Felix Leid, and
  Sergey Shadrin.
\newblock Functional relations for higher-order free cumulants.
\newblock \href{https://arxiv.org/abs/2112.12184}{arXiv:2112.12184 [math.OA]},
  2021.

\bibitem{BDKLM23}
Ga\"{e}tan Borot, Norman Do, Maksim Karev, Danilo Lewa\'{n}ski, and Ellena
  Moskovsky.
\newblock Double {H}urwitz numbers: polynomiality, topological recursion and
  intersection theory.
\newblock {\em Math. Ann.}, 387(1-2):179--243, 2023.

\bibitem{bor-eyn15}
Ga\"{e}tan Borot and Bertrand Eynard.
\newblock All order asymptotics of hyperbolic knot invariants from
  non-perturbative topological recursion of {A}-polynomials.
\newblock {\em Quantum Topol.}, 6(1):39--138, 2015.

\bibitem{BHLM14}
Vincent Bouchard, Daniel Hern\'{a}ndez~Serrano, Xiaojun Liu, and Motohico
  Mulase.
\newblock Mirror symmetry for orbifold {H}urwitz numbers.
\newblock {\em J. Differential Geom.}, 98(3):375--423, 2014.

\bibitem{BKMP09}
Vincent Bouchard, Albrecht Klemm, Marcos Mari\~{n}o, and Sara Pasquetti.
\newblock Remodeling the {B}-model.
\newblock {\em Comm. Math. Phys.}, 287(1):117--178, 2009.

\bibitem{bou-mar08}
Vincent Bouchard and Marcos Mari\~{n}o.
\newblock Hurwitz numbers, matrix models and enumerative geometry.
\newblock In {\em From {H}odge theory to integrability and {TQFT}
  tt*-geometry}, volume~78 of {\em Proc. Sympos. Pure Math.}, pages 263--283.
  Amer. Math. Soc., Providence, RI, 2008.

\bibitem{che-eyn06}
Leonid Chekhov and Bertrand Eynard.
\newblock Hermitian matrix model free energy: {F}eynman graph technique for all
  genera.
\newblock {\em J. High Energy Phys.}, (3):014, 18, 2006.

\bibitem{cot-et-al17}
Jordan~S. Cotler, Guy Gur-Ari, Masanori Hanada, Joseph Polchinski, Phil Saad,
  Stephen~H. Shenker, Douglas Stanford, Alexandre Streichera, and Masaki
  Tezuka.
\newblock Black holes and random matrices.
\newblock {\em J. High Energy Phys.}, (5):Paper No. 118, 52, 2017.

\bibitem{dij-fuj-man11}
Robbert Dijkgraaf, Hiroyuki Fuji, and Masahide Manabe.
\newblock The volume conjecture, perturbative knot invariants, and recursion
  relations for topological strings.
\newblock {\em Nuclear Phys. B}, 849(1):166--211, 2011.

\bibitem{do13}
Norman Do.
\newblock Moduli spaces of hyperbolic surfaces and their {W}eil-{P}etersson
  volumes.
\newblock In {\em Handbook of moduli. {V}ol. {I}}, volume~24 of {\em Adv. Lect.
  Math. (ALM)}, pages 217--258. Int. Press, Somerville, MA, 2013.

\bibitem{do-dye-mat17}
Norman Do, Alastair Dyer, and Daniel~V. Mathews.
\newblock Topological recursion and a quantum curve for monotone {H}urwitz
  numbers.
\newblock {\em J. Geom. Phys.}, 120:19--36, 2017.

\bibitem{do-lei-nor16}
Norman Do, Oliver Leigh, and Paul Norbury.
\newblock Orbifold {H}urwitz numbers and {E}ynard-{O}rantin invariants.
\newblock {\em Math. Res. Lett.}, 23(5):1281--1327, 2016.

\bibitem{do-man14}
Norman Do and David Manescu.
\newblock Quantum curves for the enumeration of ribbon graphs and hypermaps.
\newblock {\em Commun. Number Theory Phys.}, 8(4):677--701, 2014.

\bibitem{do-nor25}
Norman Do and Paul Norbury.
\newblock A $q$-analogue of {M}irzakhani's recursion for {W}eil--{P}etersson
  volumes.
\newblock \href{https://arxiv.org/abs/2510.12431}{arXiv:2510.12431 [math.AG]},
  2025.

\bibitem{do-tay25}
Norman Do and Arlo Taylor.
\newblock Double-scaled {SYK} correlators and map enumeration.
\newblock In preparation, 2025.

\bibitem{DOSS14}
P.~Dunin-Barkowski, N.~Orantin, S.~Shadrin, and L.~Spitz.
\newblock Identification of the {G}ivental formula with the spectral curve
  topological recursion procedure.
\newblock {\em Comm. Math. Phys.}, 328(2):669--700, 2014.

\bibitem{DOPS18}
Petr Dunin-Barkowski, Nicolas Orantin, Aleksandr Popolitov, and Sergey Shadrin.
\newblock Combinatorics of loop equations for branched covers of sphere.
\newblock {\em Int. Math. Res. Not. IMRN}, (18):5638--5662, 2018.

\bibitem{eyn14}
B.~Eynard.
\newblock Invariants of spectral curves and intersection theory of moduli
  spaces of complex curves.
\newblock {\em Commun. Number Theory Phys.}, 8(3):541--588, 2014.

\bibitem{eyn-ora07a}
B.~Eynard and N.~Orantin.
\newblock Invariants of algebraic curves and topological expansion.
\newblock {\em Commun. Number Theory Phys.}, 1(2):347--452, 2007.

\bibitem{eyn-ora15}
B.~Eynard and N.~Orantin.
\newblock Computation of open {G}romov-{W}itten invariants for toric
  {C}alabi-{Y}au 3-folds by topological recursion, a proof of the {BKMP}
  conjecture.
\newblock {\em Comm. Math. Phys.}, 337(2):483--567, 2015.

\bibitem{EGGLS24}
Bertrand Eynard, Elba Garcia-Failde, Paolo Gregori, Danilo Lewa\'{n}ski, and
  Ricardo Schiappa.
\newblock Resurgent asymptotics of {J}ackiw-{T}eitelboim gravity and the
  nonperturbative topological recursion.
\newblock {\em Ann. Henri Poincar\'{e}}, 25(9):4121--4193, 2024.

\bibitem{eyn-mul-saf11}
Bertrand Eynard, Motohico Mulase, and Bradley Safnuk.
\newblock The {L}aplace transform of the cut-and-join equation and the
  {B}ouchard-{M}ari\~{n}o conjecture on {H}urwitz numbers.
\newblock {\em Publ. Res. Inst. Math. Sci.}, 47(2):629--670, 2011.

\bibitem{eyn-ora07b}
Bertrand Eynard and Nicolas Orantin.
\newblock {W}eil--{P}etersson volume of moduli spaces, {M}irzakhani's recursion
  and matrix models.
\newblock \href{https://arxiv.org/abs/0705.3600}{arXiv:0705.3600 [math-ph]},
  2007.

\bibitem{eyn-ora09}
Bertrand Eynard and Nicolas Orantin.
\newblock Topological recursion in enumerative geometry and random matrices.
\newblock {\em J. Phys. A}, 42(29):293001, 117, 2009.

\bibitem{fan-liu-zon20}
Bohan Fang, Chiu-Chu~Melissa Liu, and Zhengyu Zong.
\newblock On the remodeling conjecture for toric {C}alabi-{Y}au 3-orbifolds.
\newblock {\em J. Amer. Math. Soc.}, 33(1):135--222, 2020.

\bibitem{gia-mai-maz25}
Alessandro Giacchetto, Pronobesh Maity, and Edward~A. Mazenc.
\newblock Matrix correlators as discrete volumes of moduli space i: Recursion
  relations, the {BMN}-limit and {DSSYK}.
\newblock \href{https://arxiv.org/pdf/2510.17728}{arXiv:2510.17728 [hep-th]},
  2025.

\bibitem{GJKS15}
Jie Gu, Hans Jockers, Albrecht Klemm, and Masoud Soroush.
\newblock Knot invariants from topological recursion on augmentation varieties.
\newblock {\em Comm. Math. Phys.}, 336(2):987--1051, 2015.

\bibitem{JKMS23}
Daniel~Louis Jafferis, David~K. Kolchmeyer, Baur Mukhametzhanov, and Julian
  Sonner.
\newblock {J}ackiw-{T}eitelboim gravity with matter, generalized eigenstate
  thermalization hypothesis, and random matrices.
\newblock {\em Phys. Rev. D}, 108:066015, Sep 2023.

\bibitem{kaz-zog15}
Maxim Kazarian and Peter Zograf.
\newblock Virasoro constraints and topological recursion for {G}rothendieck's
  dessin counting.
\newblock {\em Lett. Math. Phys.}, 105(8):1057--1084, 2015.

\bibitem{kit15}
Alexei Kitaev.
\newblock A simple model of quantum holography (parts 1 and 2).
\newblock Entanglement in Strongly-Correlated Quantum Matter, 2015.
\newblock \url{https://online.kitp.ucsb.edu/online/entangled15/}.

\bibitem{kon92}
Maxim Kontsevich.
\newblock Intersection theory on the moduli space of curves and the matrix
  {A}iry function.
\newblock {\em Comm. Math. Phys.}, 147(1):1--23, 1992.

\bibitem{kon-soi18}
Maxim Kontsevich and Yan Soibelman.
\newblock Airy structures and symplectic geometry of topological recursion.
\newblock In {\em Topological recursion and its influence in analysis,
  geometry, and topology}, volume 100 of {\em Proc. Sympos. Pure Math.}, pages
  433--489. Amer. Math. Soc., Providence, RI, 2018.

\bibitem{mir07a}
Maryam Mirzakhani.
\newblock Simple geodesics and {W}eil-{P}etersson volumes of moduli spaces of
  bordered {R}iemann surfaces.
\newblock {\em Invent. Math.}, 167(1):179--222, 2007.

\bibitem{mir07b}
Maryam Mirzakhani.
\newblock Weil-{P}etersson volumes and intersection theory on the moduli space
  of curves.
\newblock {\em J. Amer. Math. Soc.}, 20(1):1--23, 2007.

\bibitem{nor10}
Paul Norbury.
\newblock Counting lattice points in the moduli space of curves.
\newblock {\em Math. Res. Lett.}, 17(3):467--481, 2010.

\bibitem{nor13}
Paul Norbury.
\newblock String and dilaton equations for counting lattice points in the
  moduli space of curves.
\newblock {\em Trans. Amer. Math. Soc.}, 365(4):1687--1709, 2013.

\bibitem{nor20}
Paul Norbury.
\newblock Enumerative geometry via the moduli space of super riemann surfaces.
\newblock \href{https://arxiv.org/abs/2005.04378}{arXiv:2005.04378 [math.AG]},
  2020.

\bibitem{nor-sco13}
Paul Norbury and Nick Scott.
\newblock Polynomials representing {E}ynard-{O}rantin invariants.
\newblock {\em Q. J. Math.}, 64(2):515--546, 2013.

\bibitem{nor-sco14}
Paul Norbury and Nick Scott.
\newblock Gromov-{W}itten invariants of {$\Bbb{P}^1$} and {E}ynard-{O}rantin
  invariants.
\newblock {\em Geom. Topol.}, 18(4):1865--1910, 2014.

\bibitem{oku23}
Kazumi Okuyama.
\newblock Discrete analogue of the {W}eil-{P}etersson volume in double scaled
  {SYK}.
\newblock {\em J. High Energy Phys.}, (9):Paper No. 133, 15, 2023.

\bibitem{sac-ye93}
Subir Sachdev and Jinwu Ye.
\newblock Gapless spin-fluid ground state in a random quantum {H}eisenberg
  magnet.
\newblock {\em Phys. Rev. Lett.}, 70:3339--3342, May 1993.

\bibitem{wit91}
Edward Witten.
\newblock Two-dimensional gravity and intersection theory on moduli space.
\newblock In {\em Surveys in differential geometry ({C}ambridge, {MA}, 1990)},
  pages 243--310. Lehigh Univ., Bethlehem, PA, 1991.

\end{thebibliography}

\end{document}